\documentclass[12pt]{amsart}
\usepackage{amssymb,amsmath}
\usepackage[mathscr]{euscript}
\usepackage{color}
\numberwithin{equation}{section}

\newtheorem{Def}{Definition}[section]
\newtheorem{Thm}{Theorem}[section]
\newtheorem{Prop}{Proposition}[section]

\newtheorem{Lem}{Lemma}[section]

\def \eps{{\epsilon}}

\def \AA{{\mathcal A}}
\def \BB{{\mathcal B}}
\def \CC{{\mathcal C}}

\def \EE{{\mathcal E}}
\def \FF{{\mathcal F}}

\def \MM{{\mathcal M}}

\def \RR{{\mathcal R}}
\def \SS{{\mathcal S}}
\def \TT{{\mathcal T}}

\def \XX{{\mathcal X}}
\def \YY{{\mathcal Y}}

\def \ASf{{\mathsf A}}

\def \aSf{{\mathsf a}}
\def \bSf{{\mathsf b}}

\def \inttwo{\int \!\!\! \int }

\def \dee{{\mathrm d}}
 
\def \domega{{\dee \omega}}

\def \ds{{\dee s}}
\def \dt{{\dee t}}
\def \dx{{\dee x}}
\def \dv{{\dee v}}

\def \del{{\partial}}
\def \GRAD{\nabla_{\!\!x}}

\def \DIV{\nabla_{\!\!x} \! \cdot }

\def \CROSS{{\,\times\,}}
\def \DOT{{\,\cdot\,}}

\def \BAR{\overline}

\def \<{\langle}
\def \>{\rangle}

\def\R{{\mathbb R}}      

\def\D{{\mathrm{D}}}

\def\RD{{{\mathbb R}^{\D}}}
\def\SD{{{\mathbb S}^{\D-1}}}
\def\TD{{{\mathbb T}^{\D}}}

\def\Rp{{\R_+}}

\def\RDR{{\RD \! \times \R}}
\def\RDRD{{\RD \! \times \RD}}
\def\RDRDR{{\RD \! \times \RD \! \times \R}}

\def\SDRD{{\SD \! \times \RD}}

\def \init{{\mathrm{in}}}

\begin{document}

\title[Boltzmann equation near global Mawellians]{Global Solutions of the Boltzmann Equation over $\RD$\\ near Global Maxwellians with Small Mass}
 
\author[C. Bardos]{Claude Bardos}
\address[C.B.]{Laboratoire J.-L. Lions, BP187, 75252 Paris Cedex 05, France}
\email{claude.bardos@gmail.com}

\author[I.M. Gamba]{Irene M. Gamba}
\address[I. G.]{University of Texas, ICES and Dept. of Mathematics. Austin, TX 78712-1082, USA}
\email{gamba@math.utexas.edu}

\author[F. Golse]{Fran\c cois Golse}
\address[F.G.]{Ecole Polytechnique, CMLS, 91128 Palaiseau Cedex, France}
\email{francois.golse@math.polytechnique.fr}

\author[C.D. Levermore]{C. David Levermore}
\address[C.D. L.]{University of Maryland, IPST and Dept. of Mathematics, College Park, MD 20742-2431, USA}
\email{lvrmr@math.umd.edu}

\begin{abstract}
We study the dynamics defined by the Boltzmann equation set in the Euclidean space $\RD$ in the vicinity of global Maxwellians with finite mass. A global Maxwellian is a special solution of the Boltzmann equation for 
which the collision integral vanishes identically. In this setting, the dispersion due to the advection operator quenches the dissipative effect of the Boltzmann collision integral. As a result, the large time limit of solutions
of the Boltzmann equation in this regime is given by noninteracting, freely transported states and can be described with the tools of scattering theory.
\end{abstract}



\maketitle
 
\section{Introduction}\label{S-Intro}

In kinetic theory, the state of a monatomic gas is described by its distribution function $F\equiv F(v,x,t)\ge 0$, that is the number density at time $t$ of gas molecules with velocity $v\in\RD$ located at the position $x\in\RD$. 
The distribution function is governed by the Boltzmann equation
\begin{equation}
\label{BoltzEq0}
\del_tF+v\DOT\GRAD F=\BB(F,F)\,,
\end{equation}
where $\BB(F,F)$ is a quadratic integral operator acting on the $v$ variable only, known as the Boltzmann collision integral. The collision integral has a rather complicated expression whose details are not needed in this 
introduction. Suffice it to say that all the information on molecular interaction needed for the kinetic description of a gas is encoded in the collision kernel $\bSf(\omega,V)$, a nonegative function of the relative velocity $V$ 
of colliding particle pairs and of a unit vector $\omega$ that measures the deviation of relative velocity before and after collision. The explicit formula of the Boltzmann collision integral and its dependence on the collision
kernel $\bSf$ will be given in section \ref{SS-CollKer} below.

The present paper investigates the long time behavior of solutions $F\equiv F(v,x,t)$ of the Boltzmann equation (\ref{BoltzEq0}) on $\RDRD\CROSS\Rp$ satisfying some appropriate decay conditions as $|x|+|v|\to\infty$, implying
in particular that
\begin{equation}\label{DecayL1w}
\iint_\RDRD(1+|x|^2+|v|^2)F(v,x,t)\,\dv\dx<\infty\hbox{ for each }t\ge 0\,.
\end{equation}

The Boltzmann equation set in the spatial domain $\RD$ involves two very different mechanisms, namely dispersion and relaxation to local equilibrium. 

Dispersion is associated to the free transport equation, and one of its manifestations is the following observation. Let $f\equiv f(v,x,t)$ be a solution of
$$
\del_tf+v\DOT\GRAD f=0\,,\qquad f\big|_{t=0}=f^\init\,,
$$
and assume that $f^\init\equiv f^\init(v,x)$ satisfies a bound of the form
$$
0\le f^\init(v,x)\le\phi(x)\qquad\hbox{ for a.e. }(v,x)\in\RDRD\,,
$$
where $\phi\in L^1(\RD)$. Then, for each $t>0$, the macroscopic density $\rho$ associated to $f$ satisfies
$$
0\le\rho(x,t):=\int_\RD f(v,x,t)\,\dv\le\int_\RD\phi(x-tv)\,\dv=\frac1{t^\D}\int_\RD\phi(y)\,\dee y=\frac1{t^\D}\|\phi\|_{L^1}
$$
for a.e. $x\in\RD$. In particular 
$$
\rho(x,t)\to 0\quad\hbox{ as }t\to+\infty\hbox{ for a.e. }x\in\RD\,.
$$

Relaxation to local equilibrium is associated to the collision integral and can be more or less formulated as follows. Let $g\equiv g(v,t)$ be a solution of the space homogeneous Boltzmann equation
$$
\del_tg=\BB(g,g)\,,\qquad g\big|_{t=0}=g^\init\,,
$$
where $g^\init\equiv g^\init(v)$ satisfies the assumptions
$$
g^\init(v)\ge 0\hbox{ for a.e. }v\in\RD\,\quad\hbox{ and }\int_\RD g^\init(v)(1+|v|^2)\,\dv<+\infty\,.
$$
In the limit as $t\to+\infty$, we expect that $g(v,t)$ converges to a Maxwellian distribution, i.e.
$$
g(t,v)\to M[\rho^\infty,u^\infty,\theta^\infty](v):=\frac{\rho^\infty}{(2\pi\theta^\infty)^{\frac{\D}2}}\exp\left(-\frac{|v-u^\infty|^2}{2\theta^\infty}\right)\,,
$$
where 
$$
\rho^\infty:=\int_\RD g^\init (v)\,\dv\,,\quad u^\infty:=\frac1{\rho^\infty}\int_\RD g^\init (v)\,\dv\,,\quad \theta^\infty:=\frac1{\rho^\infty}\int_\RD\tfrac1\D |v-u^\infty|^2g^\init(v)\,\dv
$$
if $\rho^\infty>0$, while $g(v,t)=0$ for a.e. $(v,t)\in\RDRD$ if $\rho^\infty=0$. 

In general, dispersion and relaxation to local equilibrium are competing mechanisms, because the effect of molecular collisions at the position $x\in\RD$ obviously vanishes if the macroscopic density $\rho(x,t)\to 0$ 
as $t\to+\infty$. For instance, dispersion is used in \cite{IllnerShin,Hamda,Arsenio} to control the nonlinear collision integral, and to establish the global existence of solutions of the Boltzmann equation. 

However, these two mechanisms cooperate to produce a remarkable class of explicit solutions of the Boltzmann equation in the spatial domain $\RD$, henceforth referred to as \textit{global Maxwellians}.

\begin{Def} A global Maxwellian is a distribution function $\MM\equiv\MM(v,x,t)$ satisfying both
$$
\del_t\MM+v\DOT\GRAD\MM=0\quad\hbox{ and }\quad\BB(\MM,\MM)=0\,.
$$
\end{Def}

An example of global Maxwellian is
$$
\MM(v,x,t):=e^{-|x-tv|^2}\,.
$$

A complete description of global Maxwellians with finite mass can be found in \cite{Lvrmr}. More precisely, the main result in \cite{Lvrmr} is the following variational characterization. 

Given $F^\init\in L^1(\RDRD,(1+|x|^2+|v|^2)\dx\dv)$ such that $F^\init\ge 0$ a.e. on $\RDRD$, there exists a unique global Maxwellian $\MM_{F^\init}$ such that\footnote{For $a,b\in\RD$, the notation $a\wedge b$ 
designates the skew-symmetric tensor $a\otimes b-b\otimes a$.}
$$
\iint_{\RDRD}\left(\begin{matrix}1\\ v\\ |v|^2\\ x-tv\\ |x-tv|^2\\(x-tv)\cdot v\\ x\wedge v\end{matrix}\right)\MM_{F^\init}(v,x,t)\,\dv\dx
=
\iint_{\RDRD}\left(\begin{matrix}1\\ v\\ |v|^2\\ x\\ |x|^2\\ x\cdot v\\ x\wedge v\end{matrix}\right)F^{\init}(v,x)\,\dv\dx
$$
for all $t\in\R$. Notice that the left hand side of the equality above is independent of $t$. Indeed, any global Maxwellian $\MM$ satisfies $\MM(v,x,t)=\MM(v,x-tv,0)$ for all $(v,x,t)\in\RDRDR$.

Moreover, for each $t\in\R$, the function $(v,x)\mapsto\MM_{F^\init}(v,x,t)$ satisfies the following variational property:
$$
H[f]:=\iint_{\RDRD}f\ln f(v,x)\,\dv\dx\ge H[\MM_{F^\init}(t)]=H[\MM_{F^\init}(0)]
$$
for all a.e. nonnegative $f\in L^1(\RDRD,(1+|x|^2+|v|^2)\dx\dv)$ such that
$$
\iint_{\RDRD}\left(\begin{matrix}1\\ v\\ |v|^2\\ x-tv\\ |x-tv|^2\\ (x-tv)\cdot v\\ x\wedge v\end{matrix}\right)f(v,x)\,\dv\dx
=\iint_{\RDRD}\left(\begin{matrix}1\\ v\\ |v|^2\\ x\\ |x|^2\\ x\cdot v\\ x\wedge v\end{matrix}\right)F^{\init}(v,x)\,\dv\dx\,,
$$
with equality if and only if $f(v,x)=\MM_{F^\init}(v,x,t)$ for a.e. $(v,x)\in\RDRD$. (See section 1.5 in \cite{Lvrmr} for more details.)

\medskip
The purpose of the present work is to study the dynamics defined by the Boltzmann equation near global Maxwellians in the Euclidean space $\RD$. In particular, we establish the existence and uniqueness of solutions in that 
regime and analyze in detail the large time behavior of these solutions. More precisely, solutions $F$ of the Boltzmann equation (\ref{BoltzEq0}) over $\RDRD\CROSS\Rp$ such that $F\Big|_{t=0}=F^\init$ satisfy 
the global conservation laws
$$
\iint_{\RDRD}\left(\begin{matrix}1\\ v\\ |v|^2\\ x-tv\\ |x-tv|^2\\ (x-tv)\cdot v\\ x\wedge v\end{matrix}\right)F(v,x,t)\,\dv\dx
=\iint_{\RDRD}\left(\begin{matrix}1\\ v\\ |v|^2\\ x\\ |x|^2\\ x\cdot v\\ x\wedge v\end{matrix}\right)F^\init(v,x)\,\dv\dx
$$
for all $t\ge 0$, under some appropriate decay condition on $F$ that implies (\ref{DecayL1w}). (See Theorem B below for a precise statement.) On the other hand, under a decay condition more stringent than (\ref{DecayL1w}), 
Boltzmann's H Theorem (see below) asserts that the function 
$$
t\mapsto H[F(t)]
$$
is nonincreasing for each solution $F$ of the Boltzmann equation (\ref{BoltzEq0}).

Together with the variational chacterization of global Maxwellians recalled above, this raises the following question, which is at the core of the present paper.

\smallskip
\noindent
\textbf{Problem.} \textit{Let $F$ be a solution of the Cauchy problem for the Boltzmann equation (\ref{BoltzEq0}) with initia data $F^\init$ satisfying appropriate decay conditions at infinity, implying in particular (\ref{DecayL1w}).
In the limit as $t\to+\infty$, does $F(t)$ converge (in some sense) to the state of maximal entropy (or minimal $H$-function) compatible with the global conservation laws satisfied by $F$? In particular, does $H[F(t)]$ converge 
to $H[\MM_{F^\init}]$ as $t\to+\infty$?}

\smallskip
The main result in the present work is that this question is answered in the negative.

\section{Main results}\label{S-Main}

\subsection{Background on global Maxwellians}\label{GlobalMaxw}

We first recall the complete description of global Maxwellians from \cite{Lvrmr}.

\smallskip
\noindent
\textbf{Theorem A.} \textit{The family of all global Maxwellians over the spatial domain $\RD$ and belonging to $L^\infty(\R_t;L^1(\RDRD))$ is of the form
$$
\MM(v,x,t)=\frac{m}{(2\pi)^\D}\sqrt{\det(Q)}\exp(-q(v-v_0,x-x_0,t))\,,
$$
with
$$
q(v,x,t)=\tfrac12(c|v|^2+a|x-tv|^2+2b(x-tv)\cdot v+v\cdot B(x-tv))\quad\hbox{ and }Q=(ac-b^2)I+B^2\,,
$$
where $m,a,c>0$, $b\in\R$, $x_0,v_0\in\RD$ and $B$ is a skew-symmetric $\D\CROSS\D$ matrix with real entries such that the symmetric matrix $Q$ is definite positive.}

\smallskip
Henceforth we denote by $\Omega$ the set
$$
\Omega:=\{(a,b,c,B)\in\R^3\CROSS\R^{\D\times\D}\,\,\hbox{Ês.t. }a,c>0\,,\,\,B=-B^T\,,\,\,\hbox{ and }(ac-b^2)I+B^2>0\}\,.
$$

\smallskip
With the notation
$$
M[\rho,u,\theta](v):=\frac{\rho}{(2\pi\theta)^{\frac{\D}2}}e^{-\frac{|v-u|^2}{2\theta}}\,,
$$
elementary computations show that
$$
\MM(v,x,t)=M[\rho(x,t),u(x,t),\theta(t)](v)\,,
$$
with
\begin{equation}\label{rho-u=}
\rho(x,t)=m\theta(t)^{\frac{\D}{2}}\sqrt{\det(\tfrac{Q}{2\pi})}\exp(-\tfrac12\theta(t)x^TQx)\,,\quad u(x,t)=\theta(t)(axt-bx-Bx)\,,
\end{equation}
and
\begin{equation}\label{theta=}
\theta(t)=\frac{1}{at^2-2bt+c}\,.
\end{equation}

\smallskip
Given any global Maxwellian $\MM$ on $\RDRDR$, we consider the Banach spaces
$$
\XX_\MM:=\MM L^\infty(\RDRDR)\,,\quad\hbox{ with norm }\|F\|_\MM:=\|F/\MM\|_{L^\infty(\RDRDR)}\,,
$$
and
$$
\YY_{\MM(0)}:=\MM(0)L^\infty(\RDRD)\,,\quad\hbox{ with norm }|f|_{\MM(0)}:=\|f/\MM(0)\|_{L^\infty(\RDRD)}\,.
$$

\subsection{Assumptions on the collision kernel}\label{SS-CollKer}

Henceforth, we assume that the collision kernel has separated form, i.e.
$$
\bSf(z,\omega)=|z|^\beta\Hat\bSf(\omega\cdot n)\qquad\hbox{ with }n=z/|z|\,,
$$
and satisfies the weak cutoff condition
$$
\BAR\bSf:=\int_\SD\Hat\bSf(\omega\cdot n)\domega<\infty\,.
$$
Such a collision kernel will be said to correspond to a ``hard'' potential for the molecular interaction if $\beta\in(0,1]$, and to a ``soft'' potential if $\beta\in(-\D,0)$. The case $\beta=0$ corresponds to an assumption 
made by Maxwell in \cite{Maxw67}, and is referred to as the case of ``Maxwell molecules''. The case of hard sphere collisions is the case where $\bSf(z,\omega)=|z\cdot\omega|$. The case $\beta\in(1,2]$ is referred 
to as ``super-hard''; it does not arise from any radial, inverse power law potential and is therefore of limited physical interest.

The collision integral is defined in terms of the collision kernel as follows. For each measurable $F\equiv F(v,x,t)$ defined a.e. on $\RDRD\CROSS I$ where $I$ is an interval of $\R$ and satisfying 
\begin{equation}\label{|F|<M}
|F(v,x,t)|\le\MM(v,x,t)\quad\hbox{ for a.e. }(v,x,t)\in\RDRD\CROSS I
\end{equation}
for some global Maxwellian $\MM$, one has
$$
\BB(F,F)(v,x,t)=\iint_\SDRD(F(v',x,t)F(v'_*,x,t)-F(v,x,t)F(v_*,x,t))\bSf(v-v_*,\omega)\domega\dv_*\,.
$$
The velocities $v'$ and $v'_*$ are defined in terms of $v$, $v_*$ and $\omega$ by the formulas
$$
v'=v-(v-v_*)\cdot\omega\omega\,,\quad v'_*=v_*+(v-v_*)\cdot\omega\omega\,.
$$
These formulas give the general solution $(v',v'_*)\in\RDRD$ of the system of equations
$$
v'+v'_*=v+v_*\,,\quad|v'|^2+|v'_*|^2=|v|^2+|v_*|^2\,,
$$
where $v$ and $v_*$ are given vectors in $\RD$. Henceforth we use the notation
$$
F=F(v,x,t)\,,\quad F_*=F(v_*,x,t)\,,\quad F'=F(v',x,t)\,,\quad F'_*=F(v'_*,x,t)\,,
$$
which is customary in the literature on the Boltzmann equation. 

Since we are dealing with cutoff kernels throughout the present work, the Boltzmann collision integral can be decomposed into gain and loss terms, denoted respectively $\BB_+(F,F)$ and $\BB_-(F,F)$, and defined
as follows
$$
\begin{aligned}
\BB_+(F,F)(v,x,t)&=\iint_\SDRD F(v',x,t)F(v'_*,x,t)\bSf(v-v_*,\omega)\domega\dv_*\,,
\\
\BB_-(F,F)(v,x,t)&=\iint_\SDRD F(v,x,t)F(v_*,x,t)\bSf(v-v_*,\omega)\domega\dv_*\,.
\end{aligned}
$$
The loss term can be recast as
$$
\BB_-(F,F)(v,x,t):=F(v,x,t)\AA(F)(v,x,t)\,,
$$
with
$$
\AA(F)(v,x,t):=\inttwo_\SDRD F(v_*,x,t)\bSf(v-v_*,\omega)\domega\dv_*\,.
$$
Integrating first in the $\omega$ variable, the term $\AA(F)$ takes the form
$$
\AA(F)(v,x,t)=\BAR\bSf\int_\RD F(v_*,x,t)|v-v_*|^\beta\dv_*\,.
$$

In particular, if $F$ is a global Maxwellian as in Theorem A, one has
\begin{equation}\label{AA(MM)=}
\begin{aligned}
\AA(\MM)(v,x,t)&=\rho(x,t)\theta(t)^{\frac{\beta}2}\aSf_\beta\left(\frac{v-u(x,t)}{\sqrt{\theta(t)}}\right)
\\
&=m\BAR\bSf\sqrt{\det(\tfrac{Q}{2\pi})}\theta(t)^{\frac{\D+\beta}2}\exp(-\tfrac12\theta(t)x^TQx)\aSf_\beta\left(\frac{v-u(x,t)}{\sqrt{\theta(t)}}\right)\,,
\end{aligned}
\end{equation}
with the notation
\begin{equation}\label{aSf=}
\aSf_\beta(w):=\int_\RD|w-w_*|^\beta M[1,0,1](w_*)\dee w_*\,.
\end{equation}

In the sequel, we shall use repeatedly the following elementary estimate: for all $F,G\in\XX_\MM$, one has
\begin{equation}\label{BB-<}
|\BB_-(F,G)(v,x,t)|\le\|F\|_\MM\|G\|_\MM\BB_-(\MM,\MM)=\|F\|_\MM\|G\|_\MM\AA(\MM)\MM\,,
\end{equation}
and
\begin{equation}\label{BB+<}
\begin{aligned}
|\BB_+(F,G)(v,x,t)|&\le\|F\|_\MM\|G\|_\MM\BB_+(\MM,\MM)&
\\
&=\|F\|_\MM\|G\|_\MM\BB_-(\MM,\MM)=\|F\|_\MM\|G\|_\MM\AA(\MM)\MM\,,
\end{aligned}
\end{equation}
where the penultimate equality follows from the identity $\BB(\MM,\MM)=0$.

\subsection{Mild solutions of the Boltzmann equation and their fundamental properties}\label{SS-MildSol}

We shall henceforth use the notation $\ASf$ to designate the advection operator, i.e. $\ASf\phi=v\DOT\GRAD\phi$, which is the infinitesimal generator of the one-parameter group $e^{t\ASf}$ defined by the formula
$$
e^{t\ASf}\phi(x,v)=\phi(x+tv,v)\,.
$$

Throughout the present paper, we shall use the following notion of solution of the Boltzmann equation.

\begin{Def}\label{D-MildSol}
\smallskip
A mild solution of the Boltzmann equation is a function $F\equiv F(v,x,t)$ belonging to $L^1_{loc}(\RDRD\CROSS I)$ where $I$ is an interval of $\R$, such that $\BB(F,F)\in L^1_{loc}(\RDRD\CROSS I)$ and
$$
e^{t_2\ASf}F(v,x,t_2)=e^{t_1\ASf}F(v,x,t_1)+\int_{t_1}^{t_2}e^{s\ASf}\BB(F,F)(v,x,s)\,\ds
$$
for a.e. $(v,x)\in\RDRD$ and $t_1,t_2\in I$. In particular, $F$ is a.e. equal to a unique element of $C(I;L^1_{loc}(\RDRD))$, to which it will be henceforth identified.
\end{Def}

\smallskip
Mild solutions of the Boltzmann equations with appropriate decay condition as $|x|+|v|\to\infty$ satisfy basic conservation properties, which are summarized in the following statement.

\smallskip
\noindent
\textbf{Theorem B.} \textit{Let $F\equiv F(v,x,t)$ be a measurable function defined a.e. on $\RDRD\CROSS I$ where $I$ is an open interval of $\R$ and satisfying the bound (\ref{|F|<M}). Then}

\smallskip
\noindent
\textit{(a) for a.e. $(x,t)\in\RD\CROSS I$
$$
\int_\RD\BB(F,F)(v,x,t)\left(\begin{matrix}1\\ v\\ \tfrac12|v|^2\end{matrix}\right)\dv=0\,.
$$}

\smallskip
\textit{Assume moreover that $F$ is a mild solution of the Boltzmann equation in the sense of distributions on $\RDRD\CROSS I$. Then}

\noindent
\textit{(b) the function $F$ satisfies the global conservation laws
$$
\frac{d}{dt}\iint_\RDRD\left(\begin{matrix}1\\ v\\ \tfrac12|v|^2\\ x-tv\\ \tfrac12|x-tv|^2\\(x-tv)\cdot v\\ x\wedge v\end{matrix}\right)F(v,x,t)\dv=0
$$
in the sense of distributions on $I$.}

\smallskip
In addition, the Boltzmann equation satisfies a dissipation property, well known under the name of ``Boltzmann's H theorem'', which is recalled below.

\smallskip
\noindent
\textbf{Boltzmann's H Theorem.} \textit{Let $F\equiv F(v,x,t)$ be a measurable function defined a.e. on $\RDRD\CROSS I$ where $I$ is an open interval of $\R$ and satisfying the bound 
$$
0\le F(v,x,t)\le\MM(v,x,t)\,,\quad\hbox{ for a.e. }(v,x,t)\in\RDRD\CROSS I
$$
where $\MM$ is a global Maxwellian. Then}

\smallskip
\noindent
\textit{(a) for a.e. $(x,t)\in\RD\CROSS I$
$$
\int_\RD\BB(F,F)(v,x,t)\ln F(v,x,t)\dv\le 0\,,
$$
(b) the inequality above is an equality if and only if $\BB(F,F)=0$ a.e. on $\RDRD\CROSS I$ or, equivalently, if and only if $F$ is a local Maxwellian, i.e. there exists $\rho\equiv\rho(x,t)\ge 0$ and $\theta\equiv\theta(x,t)>0$, 
and a vector field $u\equiv u(x,t)\in\RD$ such that 
$$
F(v,x,t)=M[\rho(x,t),u(x,t),\theta(x,t)](v)\,,
$$}

\smallskip
\textit{Assume moreover that $F$ is a mild solution of the Boltzmann equation on $\RDRD\CROSS I$ satisfying the lower bound
$$
F(v,x,t)\ge\alpha\MM(v,x,t)\,,\quad\hbox{ for a.e. }(v,x,t)\in\RDRD\CROSS I\,,
$$
where $\alpha\in(0,1)$. Then}

\noindent
\textit{(c) the Boltzmann $H$ function associated to $F$ defined as
$$
H[F](t):=\iint_\RDRD F\ln F(v,x,t)\dv
$$
satisfies 
$$
\frac{dH[F]}{dt}(t)=\iint_\RDRD\BB(F,F)(v,x,t)\ln F(v,x,t)\dv\le 0
$$
in the sense of distributions on $I$.}

\smallskip
Theorem B (a) and Boltzmann's H Theorem (a)-(b) are classical properties of the Boltzmann equation, and are discussed in most books on the Boltzmann equation, for instance \cite{Cerci88,BGP}. See also Corollary 3.2 and 
Proposition 3.3 in \cite{Golse2005} for proofs based on assumptions slightly more general than those used in the present paper. Proofs of Theorem B (b) and of part (c) of the H Theorem are given in the appendix for the reader's 
convenience.

\subsection{Dispersion vs. dissipation}\label{SS-DispDissip}

Let $\MM$ be a global Maxwellian, of the form
$$
\MM(v,x,t):=\frac{m}{(2\pi)^\D}\sqrt{\det Q}e^{-q(v,x,t)}
$$
where
$$
q(v,x,t)=\tfrac12\left(\begin{matrix}v\\ x-tv\end{matrix}\right)^T\left(\begin{matrix}cI\,\,&bI-B\\ bI+B&aI\end{matrix}\right)\left(\begin{matrix}v\\ x-tv\end{matrix}\right)\,,\quad Q:=(ac-b^2)I+B^2
$$
with $(a,b,c,B)\in\Omega$.

\begin{Lem}\label{L-Mu}
Assume that the collision kernel $b$ has separated form with $\beta\in(-\D,0]$ and let $\MM$ be a global Maxwellian as in Theorem A. Then
$$
\|\AA(\MM)(t)\|_{L^\infty(\RDRD)}\le m\BAR\bSf\sqrt{\det(\tfrac{Q}{2\pi})}\aSf_\beta(0)\theta(t)^{\frac{\D+\beta}2}\,,
$$
with
$$
\aSf_\beta(0)=2^{\beta/2}\Gamma(\tfrac{\D+\beta}2)/\Gamma(\tfrac{\D}2)>0\,.
$$
Moreover, if $\beta\in(1-\D,0]$ then
\begin{equation}\label{DefMu}
\mu(\MM):=\int_\R\|\AA(\MM)(t)\|_{L^\infty(\RDRD)}\dt<\infty\,.
\end{equation}
\end{Lem}

The next  lemma shows the effect of dispersion induced by the free transport operator on the damping coefficient in the loss term of the Boltzmann collision integral. By integrating first in the time variable before taking the 
sup norm in $x$ and $v$, one gains one extra power of the relative velocity in the collision kernel $\bSf$. Therefore, this lemma applies to all cutoff collision kernels corresponding to hard as well as  soft potentials, unlike 
Lemma \ref{L-Mu}. It extends the computation on pp. 221-222 of \cite{IllnerShin} (see also formula (3.5) in \cite{Hamda}) to the larger class of global Maxwellians described in \cite{Lvrmr} and considered in the present work.

\begin{Lem}\label{L-Nu}
Assume that the collision kernel $\bSf$ has separated form with $\beta\in(1-D,1]$. Under the assumptions above, the function
$$
(v,x,t)\mapsto\int_J\AA(\MM)(v,x-tv+sv, s)\,\ds
$$
is bounded on $\RDRDR$ for each interval $J\subset\R$. Specifically, one has
$$
\left|\int_J\AA(\MM)(v,x-tv+sv, s)\,\ds\right|\le\nu(\MM)\quad\hbox{ for a.e. }(v,x,t)\in\RDRDR\,,
$$
where
\begin{equation}\label{DefNu}
\begin{aligned}
\nu(\MM):&=\sup_{(v,x,t)\in\RDRDR}\left|\int_\R\AA(\MM)(v,x-tv+sv, s)\,\ds\right|
\\
&\le\frac{m\BAR\bSf}{(2\pi)^{\D-\frac12}\sqrt{a}}\left((2\pi a)^{\D/2}+\frac{|\SD|\sqrt{\det Q}}{\beta+\D-1}\right)\,.
\end{aligned}
\end{equation}
\end{Lem}

\subsection{Existence, uniqueness and stability for the Cauchy problem}\label{SS-CauchyPbm}

Our analysis of the dynamics of the Boltzmann equation in the neighborhood of global Maxwellian begins with the following existence and uniqueness result. It states the existence and uniqueness of the mild solution of
the Cauchy problem for the Boltzmann equation under the assumption that the initial distribution function $F^\init$ is close enough to the restriction at time $t=0$ of a global Maxwellian $\MM$ which is itself small enough 
when measured in terms of the parameter $\nu(\MM)$ defined in Lemma \ref{L-Nu}.

\begin{Thm}\label{T-Etern}
Assume that the collision kernel $\bSf$ has separated form with $\beta\in(1-\D,1]$. Let $\MM$ be a global Maxwellian such that $\nu(\MM)$ defined in (\ref{DefNu}) satisfies $\nu(\MM)<\tfrac14$. 

\smallskip
\noindent
(a) For each $F^\init\in\YY_{\MM(0)}$ such that 
$$
|F^\init-\MM(0)|_{\MM(0)}<\frac{(1-4\nu(\MM))^2}{8\nu(\MM)}\,,
$$
there exists a unique mild solution $F\in\XX_\MM$ of the Boltzmann equation such that 
$$
F(0)=F^\init\quad\hbox{ and }\|F-\MM\|_\MM\le r\,,
$$
with
$$
r=\left(\frac1{4\nu(\MM)}-1\right)\left(1-\sqrt{1-\frac{8\nu(\MM)|F^\init-\MM(0)|_{\MM(0)}}{(1-4\nu(\MM))^2}}\right)\,.
$$

\smallskip
\noindent
(b) Moreover, if $\tfrac12\le 4\nu(\MM)<1$, or if $0<4\nu(\MM)<\tfrac12$ and $|F^\init-\MM(0)|_{\MM(0)}\le 1-6\nu(\MM)$, then $r\le 1$ and therefore
$$
0\le(1-r)\MM(v,x,t)\le F(v,x,t)\le(1+r)\MM(v,x,t)\quad\hbox{ for a.e. }(v,x,t)\in\RDRDR\,.
$$
\end{Thm}

Theorem \ref{T-Etern} extends earlier works, especially those of Illner and Shinbrot \cite{IllnerShin} and Hamdache \cite{Hamda}.Ê Our proof is based on the same type of fixed-point argument that was used in \cite{Hamda}, rather than
the Kaniel-Shinbrot iteration method \cite{KanielShin}  that was used in \cite{IllnerShin}. Ê(See section 6 of \cite{KanielShin} for citations of earlier uses of the fixed-point argument.) Ê We also refer to more recent papers by Toscani 
\cite{toscani1}, by Goudon \cite{goudon}, and by Alonso and Gamba \cite{Al-Ga-09}, all of which use Kaniel-Shinbrot iteration to construct solutions near global Maxwellians for the case of soft potentials.Ê Unlike Theorem \ref{T-Etern}, 
these later references do not require a smallness condition (like $\nu(\MM) < \tfrac14$) on the reference Maxwellian.Ê Otherwise, Theorem \ref{T-Etern} considers a class of collision kernels larger than that in \cite{Al-Ga-09, goudon, 
IllnerShin, toscani1} and the largest possible class of global Maxwellians, including those with rotation, whereas \cite{Al-Ga-09, goudon,IllnerShin, toscani1} consider only global Maxwellians without rotation.Ê In \cite{BGGL} we use 
Kaniel-Shinbrot iteration to extend many of the results in this paper to solutions near global Maxwellians that do not satisfy any smallness condition.

Another difference with \cite{Hamda} is statement (b), which gives a sufficient condition for the positivity of the solution so obtained. The Boltzmann equation governs the evolution of distribution functions for gas molecules. 
Therefore, solutions of the Boltzmann equation which are negative on sets of positive measure are not physically admissible.

Henceforth, the solution $F$ of the Cauchy problem with initial condition $F(0)=F^\init$ obtained in Theorem \ref{T-Etern} will be denoted 
$$
F(t)=S_tF^\init\,,\qquad t\in\R\,.
$$
In other words, $S_t$ is the one-parameter group generated by the Boltzmann equation. Notice that, under the condition on $F^\init$ in Theorem \ref{T-Etern}, the solution $S_tF^\init$ is defined for all values of the time 
variable $t$, negative as well as positive. Such solutions are referred to as ``eternal solutions''.

Although the mathematical results obtained thoughout this paper hold for positive as well as negative times, the Cauchy problem for the Boltzmann equation for positive times is of course of greater physical interest than 
its analogue for negative times. The fact that the solutions of the Cauchy problem for the Boltzmann equation obtained in Theorem \ref{T-Etern} can be extended to all negative times is a mathematical property of the
physical regime corresponding to the assumptions of Theorem \ref{T-Etern}. This observation applies to all the statements in sections \ref{SS-CauchyPbm} and \ref{SS-LongTime}, and we shall return to it later.

\smallskip
The next theorem establishes the continuous dependence of the solution $F$ of the Cauchy problem for the Boltzmann equation in terms of the initial data $F^\init$. More precisely, we show that the one-parameter group
$S_t$ is locally Lipschitz continuous on the neighborhood of $\MM(0)$ where it is defined.

\begin{Thm}\label{T-Stab}
Assume that the collision kernel $\bSf$ has separated form with $\beta\in(1-\D,1]$. Let $\MM$ be a global Maxwellian. Assume that $\nu(\MM)$ defined in (\ref{DefNu}) satisfies the condition $\nu(\MM)<\tfrac14$. Let
$F_1^\init$ and $F_2^\init\in\YY_{\MM(0)}$ be such that
$$
\eps:=\max(|F_1^\init-\MM(0)|_{\MM(0)},|F_2^\init-\MM(0)|_{\MM(0)})<\frac{(1-4\nu(\MM))^2}{8\nu(\MM)}\,.
$$
Let $F_1(t)=S_tF_1^\init$ and $F_2(t)=S_tF_2^\init$ for all $t\in\R$. Then
$$
\|F_1-F_2\|_{\MM}\le\frac{|F_1^\init-F_2^\init|_{\MM(0)}}{\sqrt{(1-4\nu(\MM))^2-8\nu(\MM)\eps}}\,.
$$
\end{Thm}

\medskip
In the case of cutoff collision kernels corresponding to soft potentials, one has the following more general stability and uniqueness result.

\begin{Thm}\label{T-StabSoft}
Assume that the collision kernel $\bSf$ has separated form with $\beta\in(1-\D,0]$. Let $F_1$ and $F_2$ be two mild solutions of the Boltzmann equation satisfying the bound
$$
|F_j(v,x,t)|\le\MM(v,x,t)\,,\quad\hbox{ for a.e. }(v,x,t)\in\RDRDR\hbox{ and }j=1,2\,.
$$
Then
$$
\|F_1-F_2\|_\MM\le|F_1(0)-F_2(0)|_{\MM(0)}e^{4\mu(\MM)}\,,
$$
where $\mu(\MM)$ is the constant defined in (\ref{DefMu}).
\end{Thm}

\smallskip
Uniqueness is a direct consequence of the last inequality: if $F_1(0)=F_2(0)$ a.e. on $\RDRD$, then $F_1=F_2$ a.e. on $\RDRDR$. The constant $\mu(\MM)$ is not optimal in the bound above, and can be replaced with
$$
\max\left(\int_0^\infty\|\AA(\MM)(t)\|_{L^\infty(\RDRD)}\dt,\int_{-\infty}^0\|\AA(\MM)(t)\|_{L^\infty(\RDRD)}\dt\right)\,,
$$
as can be seen from the proof.

\subsection{Large time behavior}\label{SS-LongTime}

In this section, we pursue our analysis of the dynamics of the Boltzmann equation near global Maxwellian with a detailed discussion of the asymptotic behavior of $S_tF^\init$ for $t\to\pm\infty$. Recall thatÊ $\ASf$Ê denotes the advection 
operator, i.e. $\ASf\phi=v\DOT\GRAD\phi$, which is the infinitesimal generator of the one-parameter group $e^{t\ASf}$ defined by the formula $e^{t\ASf}\phi(x,v)=\phi(x+tv,v)$.

Our first result in this direction is the following simple but general observation.

\begin{Thm}\label{T-ExistLim}
Assume that the collision kernel $\bSf$ has separated form with $\beta\in(1-\D,2]$. Let $F\equiv F(v,x,t)$ be a mild solution of the Boltzmann equation defined a.e. on $\RDRD\CROSS(t_0,+\infty)$ --- resp. $\RDRD\CROSS(-\infty,t_0)$)
--- for some $t_0\in\R$. Assume that, for some global Maxwellian $\MM$, defined in terms of $m>0$ and $(a,b,c,B)\in\Omega$ as in Theorem A, the solution $F$ satisfies
$$
|F(v,x,t)|\le\MM(v,x,t)
$$
for a.e. $(v,x,t)\in\RDRD\CROSS(t_0,+\infty)$ --- resp. $\RDRD\CROSS(-\infty,t_0)$. Then there exists a unique $F^{+\infty}\equiv F^{+\infty}(v,x)$ --- resp. $F^{-\infty}\equiv F^{-\infty}(v,x)$ --- such that
$$
\|F(t)-e^{-t\ASf}F^{+\infty}\|_{L^1(\RDRD)}\to 0
$$
as $t\to+\infty$ --- resp.
$$
\|F(t)-e^{-t\ASf}F^{-\infty}\|_{L^1(\RDRD)}\to 0
$$
as $t\to-\infty$. 

The functions $F^{\pm\infty}$ are given by
$$
\begin{aligned}
F^{+\infty}=F^\init+\int_0^\infty e^{s\ASf}\BB(F,F)(s)\,\ds\,, 
\\
F^{-\infty}=F^\init-\int_{-\infty}^0e^{s\ASf}\BB(F,F)(s)\,\ds\,,
\end{aligned}
$$
and satisfy the bound
$$
|F^{\pm\infty}(v,x)|\le\MM(v,x,0)\quad\hbox{ for a.e. }(v,x)\in\RDRD\,.
$$
\end{Thm}

This theorem obviously applies to the solution $F(t)=S_tF^\init$ obtained in Theorem \ref{T-Etern}, since it satisfies the bound $-(1+r)\MM(t)\le(1-r)\MM(t)\le S_tF^\init\le(1+r)\MM(t)$ for all $t\in\R$. The asymptotic states $F^{\pm\infty}$ 
so obtained obviously satisfy the bounds
$$
(1-r)\MM(v,x,0)\le F^{\pm\infty}(v,x)\le(1+r)\MM(v,x,0)\quad\hbox{ for a.e. }(v,x)\in\RDRD\,,
$$
since $e^{t\ASf}\MM(t)=\MM(0)$ and
$$
(1-r)\MM(v,x,t)\le S_tF^\init(v,x)\le(1+r)\MM(v,x,t)\quad\hbox{ for a.e. }(v,x,t)\in\RDRDR\,.
$$

\begin{Def}\label{D-TTpm}
Let $\MM$ be a global Maxwellian. Let $F^\init$ and $F^{+\infty}$ (resp. $F^{-\infty}$) be two elements of $\YY_{\MM(0)}$. We say that $F^{+\infty}=\TT^+F^\init$ (resp. $F^{-\infty}=\TT^-F^\init$) if there exists a unique mild solution $F$ 
of the Boltzmann equation on $\RDRD\times[0,+\infty)$ (resp. on $\RDRD\times(-\infty,0]$) such that $\|F(t)-e^{-t\ASf}F^{+\infty}\|_{L^1(\RDRD)}\to 0$ as $t\to+\infty$ (resp. $\|F(t)-e^{-t\ASf}F^{-\infty}\|_{L^1(\RDRD)}\to 0$ as $t\to-\infty$).
\end{Def}

This defines two operators $\TT^+$ and $\TT^-$ on subsets of $\YY_{\MM(0)}$. In view of Theorems \ref{T-Etern} and \ref{T-ExistLim} and of the remarks before the definition above, the operators $\TT^\pm$ are defined on 
$B_{\YY_{\MM(0)}}\left(\MM(0),\frac{(1-\nu(\MM))^2}{8\nu(\MM)}\right)$ provided that $\nu(\MM)<\tfrac14$, and satisfy\footnote{If $E$ is a Banach space, $B_E(x,r)$ --- resp. $\overline{B_E(x,r)}$ --- designates the open ball --- 
resp. the closed ball --- centered at $x$ with radius $r$ in $E$.} 
$$
\TT^\pm\left(B_{\YY_{\MM(0)}}\left(\MM(0),\tfrac{(1-4\nu(\MM))^2}{8\nu(\MM)}\right)\right)\subset B_{\YY_{\MM(0)}}\left(\MM(0),\tfrac1{4\nu(\MM)}-1\right)\,.
$$

Notice that the existence of the operators $\TT^\pm$ on balls of a slightly more general class of spaces analogous to $\YY_{\MM(0)}$ had been established by Hamdache \cite{Hamda}. The more general existence theorem 
above (Theorem \ref{T-ExistLim}) is new. See also formula (16.16) in \cite{Tartar} and Theorem 5.4 in \cite{Bony} for an analogous result on discrete velocity models of the kinetic theory of gases. 

Although the asymptotic behavior of solutions of the Boltzmann equation over $\RD$ for large positive time is of greater physical interest than the large negative time limit, the fact that both limits are obtained by exactly the
same mathematical arguments is an important clue.

We know that when the Boltzmann equation is set on a bounded spatial domain, assuming for instance that $x$ belongs to some periodic box, or to some bounded, connected open set of $\RD$ with smooth boundary and appropriate
boundary conditions (such as specular reflection of the gas molecules at the boundary), its solution converges, as $t\to+\infty$,  to the uniform Maxwellian state that is compatible with the initial and boundary conditions, as well as with 
the fundamental conservation laws implied by the Boltzmann equation itself (see for instance \cite{DesviVilla}). In particular, different initial data $F^\init$ may, in the case of a bounded spatial domain, lead to the same Maxwellian state 
in the long time limit.

By analogy, one might think that the asymptotic behavior for $t\to+\infty$ of any mild solution $F$ of the Boltzmann equation over $\RD$, satisfying appropriate decay conditions more stringent than (\ref{DecayL1w}) as 
$|x|+|v|\to\infty$, is given by the global Maxwellian $\MM$ with the same globally conserved quantities as $F$ --- i.e. by the state of maximal entropy compatible with the same conserved quantities as $F$. However, this 
seems unlikely, since the mathematically analogous asymptotic behavior for $t\to-\infty$ is not expected to involve the entropy. 

In fact the dynamics defined by the Boltzmann equation set in the Euclidean space $\RD$ in the long time limit is completely different from the case of a bounded domain, as shown by the next theorem.

\begin{Thm}\label{T-121}
Assume that the collision kernel $\bSf$ has separated form with $\beta\in(1-\D,0]$. Let $F_1\equiv F_1(v,x,t)$ and $F_2\equiv F_2(v,x,t)$ be mild solutions of the Boltzmann equation defined a.e. on $\RDRD\CROSS(t_0,+\infty)$ 
--- resp. $\RDRD\CROSS(-\infty,t_0)$ --- satisfying the bounds
$$
|F_1(v,x,t)|\le\MM(v,x,t)\quad\hbox{ and }\quad|F_2(v,x,t)|\le\MM(v,x,t)
$$
for a.e. $(v,x,t)\in\RDRD\CROSS(t_0,+\infty)$ --- resp. $\RDRD\CROSS(-\infty,t_0)$ --- where $\MM$ is a global Maxwellian, defined in terms of $m>0$ and $(a,b,c,B)\in\Omega$ as in Theorem A. 

Let $F_j^{\pm\infty}\equiv F_j^{\pm\infty}(v,x)$ be such that
$$
\|F_j(t)-e^{-(t-t_0)\ASf}F_j^{+\infty}\|_{L^1(\RDRD)}\to 0\quad\hbox{ for }j=1,2
$$
as $t\to+\infty$ --- resp. 
$$
\|F_j(t)-e^{-(t-t_0)\ASf}F_j^{-\infty}\|_{L^1(\RDRD)}\to 0\quad\hbox{ for }j=1,2
$$
as $t\to-\infty$. Then
$$
|F_1(t)-F_2(t)|_{\MM(t)}\le|F_1^{+\infty}-F_2^{+\infty}|_{\MM(t_0)}e^{4\mu(\MM)}
$$
for all $t>t_0$ --- resp.
$$
|F_1(t)-F_2(t)|_{\MM(t)}\le|F_1^{-\infty}-F_2^{-\infty}|_{\MM(t_0)}e^{4\mu(\MM)}
$$
for all $t<t_0$. In particular, if $F_1^{+\infty}=F_2^{+\infty}$ (resp. $F_1^{-\infty}=F_2^{-\infty}$) a.e. on $\RDRD$, then
$$
F_1(v,x,t)=F_2(v,x,t)\hbox{ for a.e. }(v,x)\in\RDRD
$$
for all $t>t_0$ --- resp. $t<t_0$.
\end{Thm}

In other words, the operators $\TT^\pm$ are one-to-one on their domains of definition in the case of soft molecular interactions: unlike in the case of the Boltzmann equation set in a bounded domain, different initial data lead 
to different asymptotic states $F^{\pm\infty}$.

Whether these operators are onto is another natural question, which is partially answered by the next theorem. 

\begin{Thm}\label{T-TTpm}
Assume that $\bSf$ has separated form with $\beta\in(1-\D,1]$. Let $\MM$ be a global Maxwellian, with $\nu(\MM)$ defined in (\ref{DefNu}) such that $\nu(\MM)<\tfrac14$. Let $F^{\pm\infty}$ satisfy
$$
F^{\pm\infty}\in B_{\YY_{\MM(0)}}\left(\MM(0),\tfrac{(1-4\nu(\MM))^2}{8\nu(\MM)}\right)\,.
$$
Then there exists a unique $F^\init_\pm\in\YY_{\MM(0)}$ satisfying $F^\init_\pm\in\overline{B_{\YY_{\MM(0)}}(\MM(0),r)}$ with
$$
r=\left(\frac1{4\nu(\MM)}-1\right)\left(1-\sqrt{1-\frac{8\nu(\MM)|F^{\pm\infty}-\MM(0)|_{\MM(0)}}{(1-4\nu(\MM))^2}}\right)\,,
$$
and
$$
\TT^\pm F^\init_\pm=F^{\pm\infty}\,.
$$
\end{Thm}

We henceforth denote 
\begin{equation}\label{TT-1=}
(\TT^\pm)^{-1}F^{\pm\infty}:=F^\init_\pm
\end{equation}
the initial data obtained in Theorem \ref{T-TTpm}. In other words, Theorem \ref{T-TTpm} defines right inverse operators $(\TT^\pm)^{-1}:\,F^{\pm\infty}\mapsto F^\init_\pm$ such that $\TT^\pm\circ(\TT^\pm)^{-1}=\mathrm{Id}$ on the ball
$B_{\YY_{\MM(0)}}\left(\MM(0),\tfrac{(1-4\nu(\MM))^2}{8\nu(\MM)}\right)$. Besides
$$
(\TT^\pm)^{-1}\left(B_{\YY_{\MM(0)}}\left(\MM(0),\tfrac{(1-4\nu(\MM))^2}{8\nu(\MM)}\right)\right)\subset B_{\YY_{\MM(0)}}\left(\MM(0),\tfrac1{4\nu(\MM)}-1\right)\,.
$$

Notice the difference between Theorems \ref{T-121}Ê and \ref{T-TTpm}. Theorem \ref{T-121}Ê establishes the one-to-one property for the operators $F^\init\mapsto F^{\pm\infty}$ possibly for large initial data, under the only assumption 
that the Boltzmann equation has a mild solution with initial data $F^\init$ that remains below some global Maxwellian. However, we do not know whether Theorem \ref{T-121} holds for hard potentials. Theorem \ref{T-TTpm} on the other 
hand holds for all cutoff kernels, for hard as well as soft potentials, and implies that the operators $F^\init\mapsto F^{\pm\infty}$ are not only one-to-one but also onto. But Theorem \ref{T-TTpm} is only a local result: it holds only in some 
neighborhood of a global Maxwellian. 

Theorems \ref{T-121} and \ref{T-TTpm} answer in the negative the question raised in the problem stated at the end of section \ref{S-Intro}. As explained above, the asymptotic behavior of solutions of the Boltzmann equation over $\RD$
for large positive time is of much greater physical interest than the large negative time limit. Nevertheless, the mathematical methods used in the proof of Theorems \ref{T-121} and \ref{T-TTpm} allow treating both limits in the same way.

In the next theorem, we discuss the continuity properties of the operators $\TT^\pm$ and of their right inverses $(\TT^\pm)^{-1}$ defined in (\ref{TT-1=}).

\begin{Thm}\label{T-ContTT}
Assume that $\bSf$ has separated form with $\beta\in(1-\D,1]$, and let $\MM$ be a global Maxwellian with $\nu(\MM)$ defined in (\ref{DefNu}) such that $\nu(\MM)<\tfrac14$, and let $0\le\eps<\frac{(1-4\nu(\MM))^2}{8\nu(\MM)}$.

\smallskip
\noindent
(a) For $F^\init_1,F^\init_2\in\overline{B_{\YY_{\MM(0)}}\left(\MM(0),\eps\right)}$ one has
$$
|\TT^\pm F^\init_1-\TT^\pm F^\init_2|_{\MM(0)}\le\frac{|F^\init_1-F^\init_2|_{\MM(0)}}{\sqrt{(1-4\nu(\MM))^2-8\nu(\MM)\eps}}\,.
$$

\smallskip
\noindent
(b) For $F^{\pm\infty}_1,F^{\pm\infty}_2\in\overline{B_{\YY_{\MM(0)}}\left(\MM(0),\eps\right)}$ one has
$$
|(\TT^\pm)^{-1}F^{\pm\infty}_1-(\TT^\pm)^{-1}F^{\pm\infty}_2|_{\MM(0)}\le\frac{|F^{\pm\infty}_1-F^{\pm\infty}_2|_{\MM(0)}}{\sqrt{(1-4\nu(\MM))^2-8\nu(\MM)\eps}}\,.
$$
\end{Thm}

\subsection{Scattering theory for the Boltzmann equation}\label{SS-Scatter}

The results obtained in the previous section imply the existence of a scattering regime for the Boltzmann equation set in the Euclidean space $\RD$, at least in the vicinity of some global Maxwellian states --- i.e. those for
which $\nu(\MM)<\tfrac14$.

In the words of P. Lax and R. Phillips \cite{LaxPhillips}, ``Scattering theory compares the asymptotic behavior of an evolving system as $t$ tends to $-\infty$ with its asymptotic behavior as $t$ tends to $+\infty$''. Chapter 2 of
\cite{LaxPhillips} defines the notion of scattering operator in terms of the translation representation of unitary groups in Hilbert  spaces. 

Since the Boltzmann equation involves entropy production via Boltzmann's H theorem, the one-parameter group $S_t$ constructed in Theorem \ref{T-Etern} is very different from a unitary group defined on a Hilbert space.
Therefore, it is a priori unclear that the concepts of scattering theory defined in such terms can be applied in the context of the kinetic theory of gases. 

There is however a scattering theory for the linear Boltzmann equation, sketched for instance in section XI.12 of \cite{ReedSimon}. The notion of scattering operator for the linear Boltzmann equation considered by M. Reed 
and B. Simon differs from the theory described in \cite{LaxPhillips} in two ways, which they summarize as follows. ``In the first place, the natural space of states is not a Hilbert space but a cone in a (non-Hilbert) vector space; 
in the second place, the equation of motion we describe defines a one-sided dynamics [\dots]'' For that reason, the definition of the wave and scattering operators in formulas (239)-(240) of \cite{ReedSimon} differs from the 
one in \cite{LaxPhillips}. 

The present section uses the same formalism as \cite{LaxPhillips}. 

\begin{Def}\label{D-SS}
Let $\MM$ be a global Maxwellian. Let $F^{+\infty}$ and $F^{-\infty}$ be two elements of $\YY_{\MM(0)}$. We say that $F^{+\infty}=\SS F^{-\infty}$ if there exists a unique mild solution $F$ of the Boltzmann equation on 
$\RDRDR$ such that $|F(v,x,t)|\le\MM(v,x,t)$ for a.e. $(v,x,t)\in\RDRDR$ for some global Maxwellian $\MM$, and 
$$
\|F(t)-e^{-t\ASf}F^{+\infty}\|_{L^1(\RDRD)}\to 0\quad\hbox{ as }t\to+\infty\,,
$$
while 
$$
\|F(t)-e^{-t\ASf}F^{-\infty}\|_{L^1(\RDRD)}\to 0\quad\hbox{ as }t\to-\infty\,.
$$
\end{Def}

\smallskip
We have put together in the next theorem the main properties of the scattering operator $\SS$.

\begin{Thm}\label{T-SSprop}
Assume that $\bSf$ has separated form with $\beta\in(1-\D,1]$, and let $\MM$ be a global Maxwellian with $\nu(\MM)$ defined in (\ref{DefNu}) such that $\nu(\MM)<\tfrac14$. 

\smallskip
\noindent
(a) For each $F^{-\infty}\in B_{\YY_{\MM(0)}}(\MM(0),\frac{(1-4\nu(\MM))^2}{8\nu(\MM)})$, there exists a unique $F^{+\infty}\in\overline{B_{\YY_{\MM(0)}}(\MM(0),r)}$ with
$$
r=\left(\frac1{4\nu(\MM)}-1\right)\left(1-\sqrt{1-\frac{8\nu(\MM)|F^{-\infty}-\MM(0)|_{\MM(0)}}{(1-4\nu(\MM))^2}}\right)\,,
$$
such that
$$
\SS F^{-\infty}=F¬^{+\infty}\,.
$$
In particular $\SS\MM(0)=\MM(0)$.

\smallskip
\noindent
(b) For each $F^{+\infty}\in B_{\YY_{\MM(0)}}(\MM(0),\frac{(1-4\nu(\MM))^2}{8\nu(\MM)})$, there exists a unique $F^{-\infty}\in\overline{B_{\YY_{\MM(0)}}(\MM(0),r)}$ with
$$
r=\left(\frac1{4\nu(\MM)}-1\right)\left(1-\sqrt{1-\frac{8\nu(\MM)|F^{+\infty}-\MM(0)|_{\MM(0)}}{(1-4\nu(\MM))^2}}\right)\,,
$$
such that
$$
\SS F^{-\infty}=F¬^{+\infty}\,.
$$
The distribution function $F^{-\infty}$ so obtained is henceforth denoted by 
$$
\SS^{-1}F^{+\infty}:=F^{-\infty}\,.
$$
In other words, the map $\SS^{-1}$ so defined is a right inverse of $\SS$, i.e. $\SS\circ(\SS^{-1})=\hbox{Id}$ on the ball $B_{\YY_{\MM(0)}}\left(\MM(0),\tfrac{(1-4\nu(\MM))^2}{8\nu(\MM)}\right)$.

\smallskip
\noindent
(c) The maps $\SS$ and $\SS^{-1}$ are locally Lipschitz continuous on $B_{\YY_{\MM(0)}}(\MM(0),\frac{(1-4\nu(\MM))^2}{8\nu(\MM)})$. More precisely, for each $\eps$ satisfying $0<\eps<\frac{(1-4\nu(\MM))^2}{8\nu(\MM)}$ and each 
$F_1^{\pm\infty},F_2^{\pm\infty}\in\overline{B_{\YY_{\MM(0)}}(\MM(0),\eps)}$, one has
$$
\begin{aligned}
|\SS F_1^{-\infty}-\SS F_2^{-\infty}|_{\MM(0)}\le\frac{|F_1^{-\infty}-F_2^{-\infty}|}{\sqrt{(1-4\nu(\MM))^2-8\nu(\MM)\eps}}\,,\quad\hbox{ and }
\\
|\SS^{-1} F_1^{+\infty}-\SS^{-1} F_2^{+\infty}|_{\MM(0)}\le\frac{|F_1^{+\infty}-F_2^{+\infty}|}{\sqrt{(1-4\nu(\MM))^2-8\nu(\MM)\eps}}\,.
\end{aligned}
$$

\smallskip
\noindent
(d) The scattering operator $\SS$ satisfies the global conservation laws
$$
\iint_{\RDRD}\left(\begin{matrix}1\\ v\\ |v|^2\\ x\\ |x|^2\\ x\cdot v\\ x\wedge v\end{matrix}\right)\SS F^{-\infty}(v,x)\,\dv\dx=\iint_{\RDRD}\left(\begin{matrix}1\\ v\\ |v|^2\\ x\\ |x|^2\\ x\cdot v\\ x\wedge v\end{matrix}\right)F^{-\infty}(v,x)\,\dv\dx
$$
for each $F^{-\infty}\in B_{\YY_{\MM(0)}}(\MM(0),\frac{(1-4\nu(\MM))^2}{8\nu(\MM)})$.

\smallskip
\noindent
(e) The scattering operator $\SS$ decreases the Boltzmann H function: if $\tfrac12<4\nu(\MM)<1$, or if $0<4\nu(\MM)\le\tfrac12$ and $|F^{-\infty}-\MM(0)|_{\MM(0)}<1-6\nu(\MM)$, then
$$
H[F^{-\infty}]\ge H[\SS F^{-\infty}]\,,
$$
with equality if and only if there exists a global Maxwellian $\tilde\MM$ such that $F^{-\infty}=\tilde\MM(0)$.

\smallskip
\noindent
(f) Assume that $\beta\le 0$. Assume moreover that $\tfrac12\le 4\nu(\MM)<1$, or that $0<4\nu(\MM)<\tfrac12$ and that $F^{-\infty}$ satisfies the condition $|F^{-\infty}-\MM(0)|_{\MM(0)}\le 1-6\nu(\MM)$. Let $\MM_{F^{-\infty}}$ be the 
global Maxwellian such that
$$
\iint_{\RDRD}\left(\begin{matrix}1\\ v\\ |v|^2\\ x\\ |x|^2\\ x\cdot v\\ x\wedge v\end{matrix}\right)F^{-\infty}(v,x)\,\dv\dx
=
\iint_{\RDRD}\left(\begin{matrix}1\\ v\\ |v|^2\\ x\\ |x|^2\\ x\cdot v\\ x\wedge v\end{matrix}\right)\MM_{F^{-\infty}}(v,x,0)\,\dv\dx\,;
$$
(see \cite{Lvrmr}). Then
$$
H[\SS F^{-\infty}]\ge H[\MM_{F^{-\infty}}]\,,
$$
with equality if and only if $F^{-\infty}=\SS F^{-\infty}=\MM_{F^{-\infty}}(0)$.
\end{Thm}

\smallskip
The existence of a scattering operator has been established by Bony \cite{Bony} in the case of the discrete velocity models of the kinetic theory of gases --- see also lecture 16 in \cite{Tartar}, especially the sentence following formula 
(16.16). However the scattering theory so obtained is noticeably different from the one constructed in the present work. For instance, the operator analogous to $\TT^+$ in \cite{Bony} is known to be discontinuous (see section 5.6 in 
\cite{Bony}). A significant difference between the real Boltzmann equation and all the discrete velocity models is of course the class of global Maxwellians studied in \cite{Lvrmr} in the former case. Indeed, no analogue of the class of 
global Maxwellians is known to exist in general for discrete velocity models.

\smallskip
One way to poetically recast our results is that the traditional argument for the heat death of the universe is wrong, at least when applied to a universe that expands forever.  That argument asserts the universe will approach its entropy 
maximizing state, which will be a cooling homogeneous state.  However our results show that an unconfined system generally does not approach its entropy maximizing state, but rather has a dispersing asymptotic state upon which 
much of its past depends through a Lipschitz continuous bijection.  In physical terms, most particles do not experience enough collisions for the entropy maximizing state to be approached before the unconfined system disperses.  Of 
course, the unconfined system modeled here by the Boltzmann equation is not a good model for the universe and we are not claiming that the fate of the universe is not a cooling homogeneous state. Rather, we are only pointing out 
that the traditional heat death argument has gaps in it. 

\section{Proofs of Lemmas \ref{L-Mu}   and  \ref{L-Nu}}

\begin{proof}[Proof of Lemma \ref{L-Mu}]
The first inequality follows from the formula (\ref{AA(MM)=}) and observing that
$$
\aSf(w)=\int_\RD|w-w_*|^\beta M[1,0,1](w_*)\dee w_*\le\int_\RD|w_*|^\beta M[1,0,1](w_*)\dee w_*=\aSf_\beta(0)
$$
because $w\mapsto M[1,0,1](w)$ and $w\mapsto|w|^\beta$ are nonincreasing functions\footnote{\textbf{Lemma.} \textit{Let $f,g\in C((0,+\infty))$ be nonincreasing on $(0,+\infty)$, and such that the function $y\mapsto f(|x-y|)g(|y|)$ 
is integrable for each $x\in\RD$. Then
$$
\int_\RD f(|x-y|)g(]y|)dy\le\int_\RD f(|y|)g(|y|)dy\quad\hbox{ for each }x\in\RD\,.
$$}

\begin{proof}
Elementary computations show that, for each $x\in\RD$, one has
$$
\begin{aligned}
\int_\RD f(|y|)g(]y|)dy-\int_\RD f(|x-y|)g(|y|)dy=&\tfrac12\int_\RD(f(|y|)-f(|x-y|))(g(|y|)-g(|x-y|))dy
\\
=&\tfrac12\int_{|x-y|>|y|}(f(|y|)-f(|x-y|))(g(|y|)-g(|x-y|))dy
\\
&+\tfrac12\int_{|x-y|<|y|}(f(|y|)-f(|x-y|))(g(|y|)-g(|x-y|))dy\,.
\end{aligned}
$$
In the first integral on the right hand side, one has both $f(|y|)-f(|x-y|)\ge 0$ and $g(|y|)-g(|x-y|)\ge 0$ since $f$ and $g$ are nonincreasing. By the same token, $f(|y|)-f(|x-y|)\le 0$ and $g(|y|)-g(|x-y|)\le 0$ in the second integral on the 
right hand side. Hence both integrals on the right hand side are nonnegative.\end{proof}} of $|w|$ since $\beta\le 0$. (The argument in the footnote follows the proof of inequality (10.2.1) on p. 261 in \cite{HLP}). Thus
$$
\int_\R\|\AA(\MM)(t)\|_{L^\infty(\RDRD)}\dt\le m\BAR\bSf\sqrt{\det(\tfrac{Q}{2\pi})}\aSf_\beta(0)\int_\R\theta(t)^{\frac{\D+\beta}2}\dt<\infty
$$
since $\theta(t)=O(1/t^2)$ as $t\to\pm\infty$ (by \ref{theta=}) and $\frac{\D+\beta}2>\tfrac12$. The explicit formula for $\aSf_\beta(0)$ comes from an explicit computation (we recall that $|\SD|=2^{\beta/2}\Gamma(\tfrac{\D}2)$).
\end{proof}

\begin{proof}[Proof of Lemma \ref{L-Nu}]
First
$$
\begin{aligned}
\int_J\AA(\MM)(v,x-tv+sv, s)\,\ds&=\BAR\bSf\int_\RD|v-v_*|^\beta\int_J\MM(v_*,x-tv+sv,s)\,\ds\dv_*
\\
&=\BAR\bSf\int_\RD|v-v_*|^\beta\int_J\MM(v_*,x-tv+s(v-v_*),0)\,\ds\dv_*\,.
\end{aligned}
$$
Indeed, since $(\del_t+v\DOT\GRAD)\MM=0$, one has 
$$
\MM(v_*,x-tv+sv,s)=\MM(v_*,x-tv+s(v-v_*),0)\,.
$$
Then
$$
\begin{aligned}
q(v_*,X+Z,0)&=\tfrac12(a|X+Z|^2+c|v_*|^2+2(X+Z)\DOT(bI+B)v_*)
\\
&=\tfrac12(a|X+Z+\tfrac1a(bI+B)^*v_*|^2-\tfrac1a|(bI+B)^Tv_*|^2+c|v_*|^2)\,.
\end{aligned}
$$
Since $B=-B^T$, one has
$$
|(bI+B)^Tv_*|^2=((bI+B)^Tv_*|(bI+B)^Tv_*)=b^2|v_*|^2+(B^Tv_*|B^Tv_*)=b^2|v_*|^2-(B^2v_*|v_*)\,,
$$
so that
$$
\begin{aligned}
q(v_*,X+Z,0)&=\tfrac12\left(a\left|X+Z+\frac1a(bI+B)^*v_*\right|^2+\frac1{a}(ac|v_*|^2-b^2|v_*|^2+(B^2v_*|v_*))\right)
\\
&=\tfrac12\left(a\left|X+Z+\frac1a(bI+B)^Tv_*\right|^2+\frac1{a}(Qv_*|v_*)\right)\,.
\end{aligned}
$$
Thus
$$
\begin{aligned}
\left|\int_J\MM(v_*,x-tv+s(v-v_*),0)\,\ds\right|&
\\
=\frac{m}{(2\pi)^\D}\sqrt{\det Q}e^{-(Qv_*|v_*)/2a}\left|\int_Je^{-a|x-tv+\frac1a(bI+B)^Tv_*+s(v-v_*)|^2/2}\,\ds\right|&
\\
\le\frac{m}{(2\pi)^\D}\sqrt{\det Q}e^{-(Qv_*|v_*)/2a}e^{-a(|x-tv+\frac1a(bI+B)^Tv_*|^2-(x-tv+\frac1a(bI+B)^Tv_*|v-v_*)^2/|v-v_*|^2)}&
\\
\times\int_\R e^{-as^2|v-v_*|^2/2}\,\ds&
\\
=\frac{m}{(2\pi)^\D}\sqrt{\det Q}e^{-(Qv_*|v_*)/2a}e^{-a(|x-tv+\frac1a(bI+B)^Tv_*|^2-(x-tv+\frac1a(bI+B)^Tv_*|v-v_*)^2/|v-v_*|^2)}&
\\
\times\sqrt{\frac{2\pi}{a}}\frac1{|v-v_*|}&
\\
\le\frac{m}{(2\pi)^\D}\sqrt{\frac{2\pi}{a}}\frac1{|v-v_*|}\sqrt{\det Q}e^{-(Qv_*|v_*)/2a}&\,.
\end{aligned}
$$
We conclude that
$$
\begin{aligned}
\left|\int_J\AA(\MM)(v,x-tv+sv, s)\,\ds\right|&\le\frac{m\BAR\bSf}{(2\pi)^\D}\sqrt{\frac{2\pi}{a}}\sqrt{\det Q}\int_\RD|v-v_*|^{\beta-1}e^{-(Qv_*|v_*)/2a}\,\dv_*
\\
&\le\frac{m\BAR\bSf}{(2\pi)^\D}\sqrt{\frac{2\pi}{a}}\sqrt{\det Q}\int_{|v-v_*|>1}e^{-(Qv_*|v_*)/2a}\,\dv_*
\\
&+\frac{m\BAR\bSf}{(2\pi)^\D}\sqrt{\frac{2\pi}{a}}\sqrt{\det Q}\int_{|v-v_*|<1}|v-v_*|^{\beta-1}\,\dv_*
\\
&\le\frac{m\BAR\bSf}{(2\pi)^\D}\sqrt{\frac{2\pi}{a}}\left((2\pi a)^{\D/2}+\frac{|\SD|\sqrt{\det Q}}{\beta+\D-1}\right)\,.
\end{aligned}
$$
\end{proof}

\section{Existence, uniqueness and stability for the Cauchy problem}

In the case of hard potentials, $\AA(\MM)\notin L^\infty(\RDRDR)$; hence $\BB_\pm$ does not map $\XX_\MM\times\XX_\MM$ into $\XX_\MM$. Controling the collision integral requires integrating first along the characteristic 
lines of the free transport operator and using the dispersion effect of the free transport operator as in the previous section (Lemma \ref{L-Nu}), as explained in the next lemma.

\begin{Lem}\label{L-CCpm}
Consider the maps defined on $\XX_\MM\times\XX_\MM$ by
$$
\CC_\pm:\,(F,G)\mapsto\CC_\pm(F,G)(v,x,t):=\int_0^t\BB_\pm(F,G)(v,x-tv+sv,s)\,\ds\,,
$$
and set 
$$
\CC:=\CC_+-\CC_-\,.
$$
Then, for all $F,G\in\XX_\MM$, one has
$$
\|\CC(F,F)-\CC(G,G)\|_\MM\le 2\nu(\MM)\|F+G\|_\MM\|F-G\|_\MM\,,
$$
and
$$
\|\CC(F,F)\|_\MM\le 2\nu(\MM)(2+\|F-\MM\|_\MM)\|F-\MM\|_\MM\,.
$$
\end{Lem}

\begin{proof}
Observe that, for all $F,G\in\XX_\MM$,
$$
\CC_\pm(F,F)-\CC_\pm(G,G)=\tfrac12\CC_\pm(F+G,F-G)+\tfrac12\CC_\pm(F-G,F+G)\,.
$$
Therefore, by (\ref{BB+<}) and (\ref{BB-<})
$$
\begin{aligned}
|\CC_\pm(F,F)-\CC_\pm(G,G)|&\le\|F+G\|_\MM\|F-G\|_\MM
\\
&\quad\times\left|\int_0^t\AA(\MM)(v,x-tv+sv,s)\MM(v,x-tv+sv,s)\,\ds\right|
\\
&\le\|F+G\|_\MM\|F-G\|_\MM\MM(v,x,t)
\\
&\quad\times\left|\int_0^t\AA(\MM)(v,x-tv+sv,s)\,\ds\right|\,,
\end{aligned}
$$
so that
$$
\|\CC_\pm(F,F)-\CC_\pm(G,G)\|_\MM\le\nu(\MM)\|F+G\|_\MM\|F-G\|_\MM\,.
$$
With the definition of $\CC$ in terms of $\CC_\pm$, this gives the first inequality in the lemma. The second inequality in the lemma follows from the first inequality, the fact that
$$
\|F+\MM\|_\MM\le\|F-\MM\|_\MM+2\|\MM\|_\MM=\|F-\MM\|_\MM+2
$$
and the identity $\CC(\MM,\MM)=0$.
\end{proof}

\smallskip
The proof of Theorem \ref{T-Etern} is based on the previous lemma and a fixed point argument.

\begin{proof}[Proof of Theorem \ref{T-Etern}]
To say that $F$ is a (mild) solution of the Boltzmann equation
$$
(\del_t+v\DOT\GRAD)F=\BB(F,F)\,,\qquad F\Big|_{t=0}=F^\init\,,
$$
means that
$$
F(v,x,t)=F^\init(v,x-tv)+\CC(F,F)(v,x,t)\,,
$$
or in other words, that $F$ is a fixed point of the map
$$
\EE:\,G\mapsto F^\init(v,x-tv)+\CC(G,G)(v,x,t)\,.
$$
Assume that $F^\init\in\YY_{\MM(0)}$; then
$$
\EE(F)(v,x,t)-\MM(v,x,t)=F^\init(v,x-tv)-\MM(v,x-tv,0)+\CC(F,F)(v,x,t)
$$
so that
$$
\begin{aligned}
\|\EE(G)-\MM\|_\MM&\le\|F^\init/\MM(0)-1\|_{L^\infty(\RDRD)}+\|\CC(G,G)\|_\MM
\\
&\le|F^\init-\MM(0)|_{\YY_{\MM(0)}}+2\nu(\MM)(\|G-\MM\|^2_\MM+2\|G-\MM\|_\MM)\,.
\end{aligned}
$$

Set $\eps(\MM,r):=(1-4\nu(\MM)(1+\tfrac12r))r$. Obviously, $\eps(\MM,r)<r$, and since $4\nu(\MM)<1$, one has $\eps(\MM,r)\ge 2\nu(\MM)r^2>0$. Thus, if $|F^\init-\MM(0)|_{\YY_{\MM(0)}}\le\eps(\MM,r)$ and if $\|G-\MM\|_\MM\le r$, 
then
$$
\|\EE(G)-\MM\|_\MM\le\eps(\MM,r)+2\nu(\MM)(r^2+2r)=r\,.
$$

Moreover, if $F,G\in\overline{B(\MM,r)}\subset\XX_\MM$, then
$$
\begin{aligned}
\|\EE(F)-\EE(G)\|_\MM&\le\|\CC(F,F)-\CC(G,G)\|_\MM\le 2\nu(\MM)\|F+G\|_\MM\|F-G\|_\MM
\\
&\le 4\nu(\MM)(\|\tfrac12(F+G)-\MM\|_\MM+\|\MM\|_\MM)\|F-G\|_\MM
\\
&=4\nu(\MM)(r+1)\|F-G\|_\MM\,.
\end{aligned}
$$
Since $4\nu(\MM)(r+1)<1$, the map $\EE$ is a strict contraction from the closed ball $\overline{B(\MM,r)}$ of the Banach space $\XX_\MM$ into itself, and therefore has a unique fixed point in that ball. In other words, the Boltzmann 
equation with initial data $F^\init$ has a unique mild solution which belongs to the ball $\overline{B(\MM,r)}\subset\XX_\MM$. 

The map $r\mapsto\eps(\MM,r)$ is increasing on $(0,\frac1{4\nu(\MM)}-1)$ and 
$$
\sup_{0<4\nu(\MM)(r+1)<1}\eps(\MM,r)=\frac{(1-4\nu(\MM))^2}{8\nu(\MM)}\,.
$$
Thus, for each $F^\init\in\YY_{\MM(0)}$ such that $|F^\init-\MM(0)|_{\MM(0)}<\frac{(1-4\nu(\MM))^2}{8\nu(\MM)}$, there exists a unique $r~\in[0,\frac1{4\nu(\MM)}-1)$ such that $\eps(\MM,r)=|F^\init-\MM(0)|_{\MM(0)}$, which is given
by 
$$
r=\left(\frac1{4\nu(\MM)}-1\right)\left(1-\sqrt{1-\frac{8\nu(\MM)|F^\init-\MM(0)|_{\MM(0)}}{(1-4\nu(\MM))^2}}\right)\,.
$$
This proves statement (a).

If $\tfrac12\le 4\nu(\MM)<1$, then $r\le\frac1{4\nu(\MM)}-1\le 1$. If $0< 4\nu(\MM)<\tfrac12$ and $r\in[0,\frac1{4\nu(\MM)}-1]$, then $\eps(\MM,r)\le\eps(\MM,1)=1-6\nu(\MM)$ if and only if $0\le r\le 1$, so that $\|F-\MM\|_\MM\le 1$ if 
$|F^\init-\MM(0)|_{\MM(0)}\le 1-6\nu(\MM)$. This proves statement (b).
\end{proof}

\begin{proof}[Proof of Theorem \ref{T-Stab}]
One has
$$
(F_1-F_2)(v,x,t)=(F_1^\init-F_2^\init)(v,x-tv,t)+(\CC(F_1,F_1)-\CC(F_2,F_2))(v,x,t)
$$
so that
$$
\begin{aligned}
\|F_1-F_2\|_\MM\le\|(F_1^\init-F_2^\init)/\MM(0)\|_{L^\infty(\RDRD)}+\|\CC(F_1,F_1)-\CC(F_2,F_2)\|_\MM
\\
\le|F_1^\init-F_2^\init|_{\YY_{\MM(0)}}+2\nu(\MM)\|F_1+F_2\|_\MM\|F_1-F_2\|_\MM
\\
\le|F_1^\init-F_2^\init|_{\YY_{\MM(0)}}+4\nu(\MM)(1+r)\|F_1-F_2\|_\MM\,,
\end{aligned}
$$
assuming that $\|F_j-\MM\|_\MM\le r$ for $j=1,2$. Therefore
$$
\|F_1-F_2\|_\MM\le\frac{|F_1^\init-F_2^\init|_{\YY_{\MM(0)}}}{1-4\nu(\MM)(r+1)}\,.
$$
Inserting the expression of $r$ in terms of $|F^\init-\MM(0)|_{\MM(0)}$ obtained in Theorem \ref{T-Etern} (a) in the right hand side of this inequality leads to the inequality of Theorem \ref{T-Stab}.
\end{proof}

\begin{proof}[Proof of Theorem \ref{T-StabSoft}]
Since $F_j$ is a mild solution on $\RDRDR$ of the Boltzmann equation for $j=1,2$, one has
$$
e^{t\ASf}F_j(t)=F_j(0)+\int_0^{t}e^{s\ASf}\BB(F_j(s),F_j(s))\,\ds
$$
for each $t\in\R$. The inequalities (\ref{BB-<}) and (\ref{BB+<}) and the bound on $F_j$ assumed in the statement of Theorem \ref{T-121} imply that
\begin{equation}\label{BF1-BF2}
\begin{aligned}
|\BB(F_1(s),F_1(s))-\BB(F_2(s),F_2(s))|\le&\tfrac12|\BB(F_1(s)+F_2(s),F_1(s)-F_2(s))|
\\
&+\tfrac12|\BB(F_1(s)-F_2(s),F_1(s)+F_2(s))|
\\
\le&4|F_1(s)-F_2(s)|_{\MM(s)}\AA(\MM(s))\MM(s)\,,
\end{aligned}
\end{equation}
so that
$$
\begin{aligned}
|F_1(t)-F_2(t)|_{\MM(t)}&=|e^{t\ASf}(F_1(t)-F_2(t))|_{\MM(0)}
\\
&\le|F_1(0)-F_2(0)|_{\MM(0)}+\int_0^t|e^{s\ASf}(\BB(F_1,F_1)(s)-\BB(F_2,F_2)(s))|_{\MM(0)}\,\ds
\\
&=|F_1(0)-F_2(0)|_{\MM(0)}+\int_0^t|\BB(F_1,F_1)(s)-\BB(F_2,F_2)(s)|_{\MM(s)}\,\ds
\\
&\le|F_1(0)-F_2(0)|_{\MM(0)}+4\int_0^t|F_1(s)-F_2(s)|_{\MM(s)}\|\AA(\MM(s))\|_{L^\infty(\RDRD)}\,\ds\,.
\end{aligned}
$$

Using Lemma \ref{L-Mu} and applying Gronwall's inequality shows that
$$
|F_1(t)-F_2(t)|_{\MM(t)}\le|F_1(0)-F_2(0)|_{\MM(0)}\exp\left(4\left|\int_0^t\|\AA(\MM(s))\|_{L^\infty(\RDRD)}\ds\right|\right)\,,
$$
from which the announced conclusion immediatly follows.
\end{proof}

\section{Large time behavior: proofs of Theorems \ref{T-ExistLim} and \ref{T-121}}

\begin{proof}[Proof of Theorem \ref{T-ExistLim}]
Observe first that
$$
\int_\RD\AA(\MM)\MM(v,x,t)\dv=C_\aSf\theta(t)^{\frac{\beta}2}\rho(x,t)^2\,,
$$
where
$$
C_\aSf=\iint_\RDRD|w-w_*|^\beta M[1,0,1](w)M[1,0,1](w_*)\dee w\dee w_*<\infty
$$
since $\beta\in(1-\D,2]$. On the other hand
$$
\int_\RD\rho(x,t)^2\,\dx=m^2\theta(t)^\D\det(\tfrac{Q}{2\pi})\int_\RD\exp(-\theta(t)x^TQx))\dx
=
m^2\sqrt{\det(\tfrac{Q}{4\pi})}\theta(t)^{\frac{\D}2}
$$
so that
$$
\iiint_\RDRDR\AA(\MM)\MM(v,x,t)\dv\dx\dt=C_\aSf m^2\sqrt{\det(\tfrac{Q}{4\pi})}\int_\R\theta(t)^{\frac{\D+\beta}2}\dt<\infty
$$
since $\theta(t)=O(1/t^2)$ as $t\to\pm\infty$ (as implied by (\ref{theta=})) and $\frac{\D+\beta}2>\tfrac12$.

Thus, if $F$ is a mild solution of the Boltzmann equation on $\RDRD\CROSS I$ (where $I$ is some interval of $\R$) satisfying the bound in Theorem \ref{T-ExistLim}, one has
\begin{equation}\label{B(F)<A(M)M}
\begin{aligned}
|\BB(F,F)|&\le\BB_+(|F|,|F|)+\BB_-(|F|,|F|)
\\
&\le\BB_+(\MM,\MM)+\BB_-(\MM,\MM)=2\AA(\MM)\MM\,,
\end{aligned}
\end{equation}
and therefore
$$
\iiint_{\RDRD\CROSS I}|\BB(F,F)(v,x,t)|\dv\dx\dt\le 2\iiint_\RDRDR\AA(\MM)\MM(v,x,t)\dv\dx\dt<\infty\,.
$$

If $I=[t_0,+\infty)$, then
$$
e^{(t-t_0)\ASf}F(t)=F(t_0)+\int_{t_0}^te^{(s-t_0)\ASf}\BB(F(s),F(s))\dee s\,,
$$
and since
$$
\int_{t_0}^{+\infty}\|e^{(s-t_0)\ASf}\BB(F(s),F(s))\|_{L^1(\RDRD)}\dee s=\int_{t_0}^{+\infty}\|\BB(F(s),F(s))\|_{L^1(\RDRD)}\dee s<\infty\,,
$$
we conclude that $e^{(t-t_0)\ASf}F(t)$ converges to a limit in $L^1(\RDRD)$ as $t\to+\infty$. The case where $I=(-\infty,t_0]$ is handled similarly.
\end{proof}

\begin{proof}[Proof of Theorem \ref{T-121}]
Since $F_j$ is a mild solution on $\RDRD\CROSS[t_0,+\infty)$ of the Boltzmann equation for $j=1,2$, one has
$$
e^{(t'-t_0)\ASf}F_j(t')=e^{(t-t_0)\ASf}F_j(t)+\int_t^{t'}e^{(s-t_0)\ASf}\BB(F_j(s),F_j(s))\,\ds
$$
for each $t'>t>t_0$. Letting $t'\to+\infty$ and using the assumption on the large time behavior of $F_1$ and $F_2$ shows that
$$
\begin{aligned}
e^{(t-t_0)\ASf}(F_1(t)-F_2(t))=F_1^{+\infty}-F_2^{+\infty}-\int_t^{+\infty}e^{(s-t_0)\ASf}(\BB(F_1(s),F_1(s))-\BB(F_2(s),F_2(s)))\,\ds\,.
\end{aligned}
$$
The inequality (\ref{BF1-BF2}) and the bound on $F_j$ assumed in the statement of Theorem \ref{T-121} imply that
$$
|\BB(F_1(s),F_1(s))-\BB(F_2(s),F_2(s))|\le 4|F_1(s)-F_2(s)|_{\MM(s)}\AA(\MM(s))\MM(s))\,,
$$
so that
$$
\begin{aligned}
|F_1(t)-F_2(t)|_{\MM(t)}&=|e^{(t-t_0)\ASf}(F_1(t)-F_2(t))|_{\MM(t_0)}
\\
&\le|F_1^{+\infty}-F_2^{+\infty}|_{\MM(t_0)}+4\int_t^{+\infty}|F_1(s)-F_2(s)|_{\MM(s)}\|\AA(\MM(s))\|_{L^\infty(\RDRD)}\dee s
\end{aligned}
$$
in view of the obvious identity $e^{(s-t_0)\ASf}\MM(s)=\MM(t_0)$. The conclusion follows from Lemma \ref{L-Mu} and Lemma \ref{L-Gronw} below.
\end{proof}

\smallskip
The second lemma is a variant of the Gronwall inequality.

\begin{Lem}\label{L-Gronw}
Let $\phi\in L^\infty([t_0,+\infty))$ satisfy the integral inequality
$$
0\le\phi(t)\le\Delta+\int_t^{+\infty}\phi(s)m(s)\ds\,,\quad\hbox{ for a.e. }t>t_0\,,
$$
where $\Delta>0$ and $m$ is a measurable function on $[t_0,+\infty)$ satisfying
$$
m(t)>0\hbox{ for a.e. }t\ge t_0\,,\quad\hbox{Êand }\int_{t_0}^{\infty}m(s)\ds<\infty\,.
$$
Then 
$$
\phi(t)\le\Delta\exp\left(\int_{t_0}^{+\infty}m(s)\ds\right)
$$
for a.e. $t\ge t_0$.
\end{Lem}

\begin{proof}
One has
$$
\frac{\phi(t)m(t)}{\displaystyle\Delta+\int_t^\infty\phi(s)m(s)\,\ds}\le m(t)\,,
$$
so that
$$
\left[-\ln\left(\Delta+\int_t^\infty\phi(s)m(s)\ds\right)\right]_{t=\tau}^{+\infty}
=
\int_{\tau}^{+\infty}\frac{\phi(t)m(t)\dt}{\displaystyle\Delta+\int_t^\infty\phi(s)m(s)\ds}
\le
\int_\tau^{+\infty}m(t)\dt\,.
$$
Taking the exponential of both sides of this inequality, one finds that
$$
\Delta+\int_\tau^\infty\phi(s)m(s)\ds\le\Delta\exp\left(\int_\tau^{+\infty}m(s)\ds\right)\,,
$$
from which the announced inequality immediatly follows.
\end{proof}

\section{The operators $\TT^\pm$: proofs of Theorems \ref{T-TTpm} and \ref{T-ContTT}}

Going back to the proof of Theorem \ref{T-ExistLim} shows that $\TT^+F^\init_+=F^{+\infty}$ if and only if there exists a unique mild solution $F$ of the Boltzmann equation such that
$$
F^{+\infty}=e^{t\ASf}F(t)+\int_t^{+\infty}e^{s\ASf}\BB(F(s),F(s))\,\ds=0\,,\quad\hbox{Êfor all }t\in\R\,,
$$
or equivalently
$$
f(t)=F^{+\infty}-\int_t^{+\infty}e^{s\ASf}\BB(e^{-s\ASf}f(s),e^{-s\ASf}f(s))\,\ds=0\,,\quad\hbox{Êfor all }t\in\R\,,
$$
with $f(t):=e^{t\ASf}F(t)$. This is an equation for the unknown $f$, which is put in the form
$$
f=F^{+\infty}-\FF_+(f)\,,
$$
with
\begin{equation}\label{DefFF+}
\FF_+(f)(v,x,t):=\int_t^{+\infty}e^{s\ASf}\BB(e^{-s\ASf}f,e^{-s\ASf}f)(v,x,s)\,\ds\,.
\end{equation}

\begin{proof}[Proof of Theorem \ref{T-TTpm}]
Proceeding as in the proof of Lemma \ref{L-CCpm}, we see that, for each $f,g\in\XX_\MM$, one has
\begin{equation}\label{LipFF+}
\|\FF_+(f)-\FF_+(g)\|_\MM\le 2\nu(\MM)\|f+g\|_\MM\|f-g\|_\MM\,,
\end{equation}
and
\begin{equation}\label{BndFF+}
\|\FF_+(f)\|_\MM\le 2\nu(\MM)(2+\|f-\MM\|_\MM)\|f-\MM\|_\MM\,.
\end{equation}
Thus 
$$
\begin{aligned}
\|F^{+\infty}-\FF_+(f)-\MM\|_\MM&\le|F^{+\infty}-\MM|_{\MM(0)}+\|\FF_+(f)\|_\MM
\\
&\le\eps(\MM,r)+2\nu(\MM)(2+r)r=r
\end{aligned}
$$
provided that $|F^{+\infty}-\MM|_{\MM(0)}\le\eps(\MM,r):=(1-4\nu(\MM)(1+\tfrac12r))r$ and $\|f-\MM\|_\MM\le r$. In other words the map $f\mapsto F^{+\infty}-\FF_+(f)$ sends the closed ball $\overline{B_{\XX_\MM}(\MM,r)}$ into itself 
provided that $|F^{+\infty}-\MM|_{\MM(0)}\le\eps(\MM,r)$.

On the other hand, if $f,g\in\overline{B_{\XX_\MM}(\MM,r)}$
\begin{equation}\label{DefFF-}
\|\FF_+(f)-\FF_+(g)\|_\MM\le 2\nu(\MM)\|f+g\|_\MM\|f-g\|_\MM\le 4\nu(\MM)(1+r)\|f-g\|_\MM\,,
\end{equation}
so that the map $f\mapsto F^{+\infty}-\FF_+(f)$ is a strict contraction on the closed ball $\overline{B_{\XX_\MM}(\MM,r)}$ provided that $4\nu(\MM)(1+r)<1$. By the Banach fixed point theorem, the map $f\mapsto F^{+\infty}-\FF_+(f)$ 
has a unique fixed point $f\in\overline{B_{\XX_\MM}(\MM,r)}$. 

Thus $F(t)=e^{-t\ASf}f(t)$ is a mild solution of the Boltzmann equation on $\RDRDR$ such that $e^{t\ASf}F(t)\to F^{+\infty}$ in $L^1(\RDRD)$ as $t\to+\infty$. Setting $F^\init_+:=F(0)$, one has $\TT^+F^\init_+=F^{+\infty}$.

Arguing as in the proof of Theorem \ref{T-Etern},
$$
|F^{+\infty}-\MM(0)|_{\MM(0)}<\tfrac{(1-4\nu(\MM))^2}{8\nu(\MM)}\Rightarrow\|f-\MM\|_\MM=\|F-\MM\|_\MM\le r\,,
$$
where $r$ is the unique element of $[0,\frac1{4\nu(\MM)}-1)$ such that $\eps(\MM,r)=|F^{+\infty}-\MM(0)|_{\MM(0)}$, i.e.
$$
r=\left(\frac1{4\nu(\MM)}-1\right)\left(1-\sqrt{1-\frac{8\nu(\MM)|F^{+\infty}-\MM(0)|_{\MM(0)}}{(1-4\nu(\MM))^2}}\right)\,.
$$
In particular $|F^\init_+-\MM(0)|_{\MM(0)}\le\|F-\MM\|_\MM\le r$.

By looking instead for a fixed point of the map $f\mapsto F^{-\infty}+\FF_-(f)$, where
$$
\FF_-(f)(v,x,t):=\int_{-\infty}^te^{s\ASf}\BB(e^{-s\ASf}f,e^{-s\ASf}f)(v,x,s)\,\ds\,,
$$
one obtains $F^\init_-$ such that $\TT^-F^\init_-=F^{-\infty}$ in the same way. Since $\FF_-$ satisfies the same estimates (\ref{LipFF+}) and (\ref{BndFF+}) as $\FF_+$, i.e.
\begin{equation}\label{LipFF-}
\|\FF_-(f)-\FF_-(g)\|_\MM\le 2\nu(\MM)\|f+g\|_\MM\|f-g\|_\MM\,,
\end{equation}
and
\begin{equation}\label{BndFF-}
\|\FF_-(f)\|_\MM\le 2\nu(\MM)(2+\|f-\MM\|_\MM)\|f-\MM\|_\MM\,.
\end{equation}
one obtains the existence and uniqueness of $F^\init_-$ satisfying the condition $|F^\init_--\MM(0)|_{\MM(0)}\le r$, provided that $|F^{-\infty}-\MM(0)|_{\MM(0)}<\frac{(1-4\nu(\MM))^2}{8\nu(\MM)}$.
\end{proof}

\smallskip
The proof of Theorem \ref{T-ContTT} (a) is based on Theorem \ref{T-Stab}, while Theorem \ref{T-ContTT} (b) follows from the classical argument proving the continuous dependence of the fixed point on the initial data.

\begin{proof}[Proof of Theorem \ref{T-ContTT}]
Since $\nu(\MM)<\tfrac14$ and $F_1^\init,F_2^\init\in\overline{B_{\YY_{\MM(0)}}(\MM(0),\eps)}$ with $0\le\eps<\frac{(1-4\nu(\MM))^2}{8\nu(\MM)}$, applying Theorem \ref{T-Stab} shows that
$$
|e^{t\ASf}F_1(t)-e^{t\ASf}F_2(t)|_{\MM(0)}\le\|F_1-F_2\|_\MM\le\frac{|F^\init_1-F^\init_2|_{\MM(0)}}{\sqrt{(1-4\nu(\MM))^2-8\nu(\MM)\eps}}
$$
for each $t\in\R$, where $F_j(t)=S_tF^\init_j$ for $j=1,2$. Observing that $e^{t\ASf}F_j(t)\to\TT^\pm F^\init_j$ in $L^1(\RDRD)$ as $t\to\pm\infty$, we conclude that
$$
|\TT^\pm F^\init_1-\TT^\pm F^\init_2|_{\MM(0)}\le\frac{|F^\init_1-F^\init_2|_{\MM(0)}}{\sqrt{(1-4\nu(\MM))^2-8\nu(\MM)\eps}}\,,
$$
which is statement (a).

If $F_1^{+\infty},F_2^{+\infty}\in\overline{B_{\YY_{\MM(0)}}(\MM(0),\eps)}$ with $0\le\eps<\frac{(1-4\nu(\MM))^2}{8\nu(\MM)}$, applying Theorem \ref{T-TTpm} shows that $F_j^{+\infty}=\TT^+F_j^\init$ for $j=1,2$, where $F_j^{\init}=f_j(0)$
and $f_j$ satisfies 
$$
f_j=F^{+\infty}_j-\FF_+(f_j)\,,\quad\hbox{ and }\|f_j-\MM\|_\MM\le r\,,\quad j=1,2
$$
where $\FF_+$ is defined in (\ref{DefFF+}). Because of (\ref{LipFF+})
$$
\begin{aligned}
\|f_1-f_2\|_\MM&\le|F^{+\infty}_1-F^{+\infty}_2|_{\MM(0)}+\|\FF_+(f_1)-\FF_+(f_2)\|_\MM
\\
&\le|F^{+\infty}_1-F^{+\infty}_2|_{\MM(0)}+4\nu(\MM)(1+r)\|f_1-f_2\|_\MM\,,
\end{aligned}
$$
so that
$$
|F^\init_1-F^\init_2|_{\MM(0)}\le\|f_1-f_2\|_\MM\le\frac{|F^{+\infty}_1-F^{+\infty}_2|_{\MM(0)}}{1-4\nu(\MM)(1+r)}\,.
$$
Inserting  the expression of $r$ in terms of $|F^{+\infty}-\MM(0)|_{\MM(0)}$ given in Theorem \ref{T-TTpm} in the right hand side of this inequality leads to the estimate in statement (b) for the operator $(\TT^+)^{-1}$. The analogous
estimate for $(\TT^-)^{-1}$ is obtained by the same argument involving $\FF_-$ defined in (\ref{DefFF-}) instead of $\FF_+$.
\end{proof}

\section{The scattering operator $\SS$: proof of Theorem \ref{T-SSprop}}

If $F$ is a mild solution of the Boltzmann equation on $\RDRDR$, 
$$
e^{t'\ASf}F(t')=e^{t\ASf}F(t)+\int_t^{t'}e^{s\ASf}\BB(F(s),F(s))\,\dx
$$
for each $t,t'\in\R$. In terms of $f(t)=e^{t\ASf}F(t)$, the equality above is recast as
$$
f(t')=f(t)+\int_t^{t'}e^{s\ASf}\BB(e^{-s\ASf}f(s),e^{-s\ASf}f(s))\,\ds\,.
$$
According to Definition \ref{D-SS}, one has $F^{+\infty}=\SS F^{-\infty}$ if the function $s\mapsto e^{s\ASf}\BB(e^{-s\ASf}f(s),e^{-s\ASf}f(s))$ is integrable on $\R$ a.e. on $\RDRD$ and
$$
F^{+\infty}=F^{-\infty}+\int_\R e^{s\ASf}\BB(e^{-s\ASf}f(s),e^{-s\ASf}f(s))\,\ds\,.
$$

\begin{proof}[Proof of Theorem \ref{T-SSprop}]
Since the functional $\FF_-$ defined in (\ref{DefFF-}) satisfies the estimates (\ref{LipFF-}) and (\ref{BndFF-}), arguing as in the proof of Theorem \ref{T-Etern} shows that the map $f\mapsto F^{-\infty}+\FF_-(f)$ sends the closed ball
$\overline{B_\XX(\MM,r)}$ into itself provided that $|F^{-\infty}-\MM|_{\MM(0)}\le\eps(\MM,r)$, where we recall that $\eps(\MM,r)=(1-4\eps(\MM)(1+\tfrac12r))r$. Since $4\nu(\MM)(1+r)<1$, one has $\eps(\MM,r)>0$. Moreover 
$$
\|\FF_-(f)-\FF_-(g)\|_\MM\le 4\nu(\MM)(1+r)\|f-g\|_\MM\quad\hbox{ for all }f,g\in\overline{B_\XX(\MM,r)}
$$
by (\ref{LipFF-}). By the Banach fixed point theorem, the map $f\mapsto F^{-\infty}+\FF_-(f)$ has a unique fixed point in $\overline{B_\XX(\MM,r)}$. The distribution function $F(t):=e^{-t\ASf}f(t)$ is therefore a mild solution of the 
Boltzmann equation on $\RDRDR$ and satisfies the bound $\|F-\MM\|_\MM\le r$, so that $|F(v,x,t)|\le(1+r)\MM(v,x,t)$ for a.e. ($v,x,t)\in\RDRDR$. By Theorem \ref{T-ExistLim}, there exists a unique $F^{+\infty}\in\YY_{\MM(0)}$ 
such that $f(t)=e^{t\ASf}F(t)\to F^{+\infty}$ as $t\to+\infty$, which means that $\SS F^{-\infty}=F^{+\infty}$. Besides $|F^{+\infty}-\MM(0)|_{\MM(0)}\le\|F-\MM\|_\MM\le r$. This defines the scattering map $\SS$ on the closed ball
$\overline{B_{\YY_{\MM(0)}}(\MM,\eps(\MM,r))}$ and shows that $\SS(\overline{B_{\YY_{\MM(0)}}(\MM,\eps(\MM,r)))}\subset\overline{B_{\YY_{\MM(0)}}(\MM,r)}$ provided that $4\nu(\MM)(1+r)<1$. 

Arguing as in the proof of Theorem \ref{T-Etern}, we conclude that $\SS$ is defined on the open ball $B_{\YY_{\MM(0)}}(\MM,\tfrac{(1-4\nu(\MM))^2}{8\nu(\MM)})$ and that $|\SS F^{-\infty}-\MM|_{\MM(0)}\le r$ where 
$$
r=\left(\frac1{4\nu(\MM)}-1\right)\left(1-\sqrt{1-\frac{8\nu(\MM)|F^{+\infty}-\MM(0)|_{\MM(0)}}{(1-4\nu(\MM))^2}}\right)\,.
$$
i.e. $r$ is the unique element of $[0,\frac1{4\nu(\MM)}-1)$ such that $\eps(\MM,r)=|F^{-\infty}-\MM(0)|_{\MM(0)}$. This proves statement (a).

Statement (b) is obtained in exactly the same manner, by seeking a fixed point of the map $f\mapsto F^{+\infty}-\FF_+(f)$, where $\FF_+$ is the map defined in (\ref{DefFF+}).

If $F_1^{-\infty},F_2^{-\infty}\in\overline{B_{\YY_{\MM(0)}}(\MM(0),\eps)}$ with $0\le\eps<\frac{(1-4\nu(\MM))^2}{8\nu(\MM)}$, there exists unique elements $f_1,f_2$ of $\XX_\MM$ such that 
$$
f_j=F^{-\infty}_j+\FF_-(f_j)\,,\quad\hbox{ and }\|f_j-\MM\|_\MM\le r\,,\quad j=1,2\,.
$$
Because of (\ref{LipFF-})
$$
\begin{aligned}
\|f_1-f_2\|_\MM&\le|F^{-\infty}_1-F^{-\infty}_2|_{\MM(0)}+\|\FF_-(f_1)-\FF_-(f_2)\|_\MM
\\
&\le|F^{-\infty}_1-F^{-\infty}_2|_{\MM(0)}+4\nu(\MM)(1+r)\|f_1-f_2\|_\MM\,,
\end{aligned}
$$
so that
$$
|\SS F ^{-\infty}_1-\SS F^{-\infty}_2|_{\MM(0)}\le\|f_1-f_2\|_\MM\le\frac{|F^{-\infty}_1-F^{-\infty}_2|_{\MM(0)}}{1-4\nu(\MM)(1+r)}\,.
$$
The analogous local Lipschitz continuity estimate for the map $\SS^{-1}$ is obtained by the same arguments involving $\FF_+$ instead of $\FF_-$. This proves statement (c).

Let $a,c,l,m\in\R$ and $b,p,q\in\RD$. For each $F^{-\infty}\in B_{\YY_{\MM(0)}}(\MM(0),\tfrac{(1-4\nu(\MM))^2}{8\nu(\MM)})$, let $f$ be the unique fixed point of the map $f\mapsto F^{-\infty}+\FF_-(f)$ in the closed ball 
$\overline{B_{\XX_\MM}(0,r)}$ where 
$$
r=\left(\frac1{4\nu(\MM)}-1\right)\left(1-\sqrt{1-\frac{8\nu(\MM)|F^{+\infty}-\MM(0)|_{\MM(0)}}{(1-4\nu(\MM))^2}}\right)\,.
$$
Thus $F(t)=e^{-t\ASf}f(t)$ is a mild solution of the Boltzmann equation. By Theorem B (b), the distribution function $F$ satisfies the global conservation law
$$
\frac{d}{dt}\iint_{\RDRD}(a+b\cdot v+c|v|^2+p\cdot(x\wedge v)+q\cdot(x\!-\!tv)+l(x\!-\!tv)\cdot v+m|x\!-\!tv|^2)F(v,x,t)\,\dv\dx=0\,.
$$
Equivalently, $f(t)=e^{t\ASf}F(t)$ satisfies
$$
\frac{d}{dt}\iint_{\RDRD}(a+b\cdot v+c|v|^2+p\cdot(x\wedge v)+q\cdot x+lx\cdot v+m|x|^2)f(v,x,t)\,\dv\dx=0\,.
$$
Since $f(t)\to F^{-\infty}$ in $L^1(\RDRD)$ as $t\to-\infty$ and $f(t)\to\SS F^{-\infty}$ in $L^1(\RDRD)$ as $t\to+\infty$ and $\|f\|_\MM\le 1+r$, one has
$$
\begin{aligned}
\iint_{\RDRD}(a+b\cdot v+c|v|^2+p\cdot(x\wedge v)+q\cdot x+lx\cdot v+m|x|^2)\SS F^{-\infty}(v,x)\,\dv\dx&
\\
=
\iint_{\RDRD}(a+b\cdot v+c|v|^2+p\cdot(x\wedge v)+q\cdot x+lx\cdot v+m|x|^2)F^{-\infty}(v,x)\,\dv\dx&\,,
\end{aligned}
$$
which is obviously equivalent to statement (d).

Let $F$ be the mild solution of the Boltzmann equation on $\RDRDR$ such that 
\begin{equation}\label{ScattF}
\|e^{t\ASf}F(t)-F^{\pm\infty}\|_{L^1(\RDRD)}\to 0\hbox{ as }t\to\pm\infty
\end{equation}
with $F^{+\infty}=\SS F^{-\infty}$, while $\|F-\MM\|_\MM\le r$ with $r$ given by the formula in statement (a). If $\tfrac12<4\nu(\MM)<1$ or if $0<4\nu(\MM)\le\tfrac12$ and $\|F^{-\infty}-\MM\|_\MM<1-6\nu(\MM)$, then $r$ given by 
the expression in statement (a) satisfies $0\le r<1$ so that 
$$
0\le(1-r)\MM(t)\le F(t)\le (1+r)\MM(t)\hbox{Êa.e. on }\RDRD\,,\quad\hbox{Êfor all }t\in\R\,.
$$
By Boltzmann's H Theorem (c), the $t\mapsto H[F(t)]$ is nonincreasing on $\R$ and one has
$$
H[F(t)]-H[F(-t)]=\int_{-t}^t\iint_\RDRD\BB(F(s),F(s))\ln F(s)\,\ds\,,\quad\hbox{ for all }t\ge 0\,.
$$
On the other hand $H[F(t)]=H[e^{t\ASf}F(t)]$, and since $(1-r)\MM(0)\le e^{t\ASf}F(t)\le(1+r)\MM(0)$ a.e. on $\RDRD$ for all $t\in\R$, one has
$$
(1-r)\MM(0)\ln((1-r)\MM(0))\le e^{t\ASf}F(t)\le(1+r)\MM(0)\ln((1+r)\MM(0))\hbox{ a.e. on }\RDRD\,.
$$
Since $\MM(0)$ and $\MM(0)\ln\MM(0)\in L^1(\RDRD)$, we conclude by dominated convergence that
$$
H[F(t)]\to H[F^{-\infty}]\hbox{Êas }t\to-\infty\,,\quad H[F(t)]\to H[\SS F^{-\infty}]\hbox{Êas }t\to+\infty\,,
$$
and
$$
H[\SS F^{-\infty}]-H[F^{-\infty}]=\int_{-\infty}^{+\infty}\iint_\RDRD\BB(F(s),F(s))\ln F(s)\,\dv\dx\ds\le 0\,,
$$
with equality if and only if 
$$
\int_\RD\BB(F,F)(v,x,s)\,dv=0\hbox{ for a.e. }(x,s)\in\RDR\,.
$$
By Boltzmann's H Theorem (b), this implies that $F$ is a.e. equal to a local Maxwellian, i.e. that $F$ is of the form $F(v,x,t)=M[\rho(x,t),u(x,t),\theta(x,t)](v)$ for a.e. $(v,x,t)\in\RDRDR$. Since $F$ is also a mild solution of the Boltzmann
equation, this local Maxwellian must be a global Maxwellian $\tilde\MM$. Hence 
$$
F^{-\infty}=\lim_{t\to-\infty}e^{t\ASf}F(t)=\lim_{t\to-\infty}e^{t\ASf}\tilde\MM(t)=\tilde\MM(0)\,.
$$
This establishes statement (e).

As for statement (f), observe that statement (d) implies that
$$
\iint_{\RDRD}\left(\begin{matrix}1\\ v\\ |v|^2\\ x-tv\\ |x-tv|^2\\ (x-tv)\cdot v\\ x\wedge v\end{matrix}\right)\MM_{F^{-\infty}}(v,x,t)\,\dv\dx
=
\iint_{\RDRD}\left(\begin{matrix}1\\ v\\ |v|^2\\ x\\ |x|^2\\ x\cdot v\\ x\wedge v\end{matrix}\right)\SS F^{-\infty}(v,x)\,\dv\dx\,.
$$
The variational characterization of $\MM_{\SS F^{-\infty}}$ (see the remark following Theorem 1.1 in \cite{Lvrmr}) implies that $\MM_{\SS F^{-\infty}}=\MM_{F^{-\infty}}$, so that $H[\SS F^{-\infty}]\ge H[\MM_{F^{-\infty}}]$, with equality 
if and only if $\SS F^{-\infty}=\MM_{F^{-\infty}(0)}$. Since $\beta\le 0$, applying Theorem \ref{T-121} shows that the only mild solution $F$ of the Boltzmann equation on $\RDRDR$ such that $|F|\le\MM$ a.e. on $\RDRDR$ for some 
global Maxwellian $\MM$ and such that $\|F(t)-e^{-t\ASf}\MM_{F^{-\infty}(0)}\|_{L^1(\RDRD)}\to 0$ as $t\to+\infty$ is $\MM_{F^{-\infty}}$. Hence
$$
F^{-\infty}=\lim_{t\to-\infty}e^{t\ASf}F(t)=\lim_{t\to-\infty}e^{t\ASf}\MM_{F^{-\infty}}(t)=\MM_{F^{-\infty}}(0)\,.
$$
\end{proof}

\section{Conclusion and perspectives}

The main results in this paper bear on the large time behavior of solutions of the Boltzmann equation set in the Euclidean space $\RD$ in the vicinity of global Maxwellians. The fact that both operators $\TT^+$ and $\TT^-$ in 
Definition \ref{D-TTpm} are locally one-to-one and onto is a major difference between the dynamics of the Boltzmann equation in the Euclidean space $\RD$ and in the torus $\TD$, or in any bounded domain with specular 
reflection of the gas molecules at the boundary. The reason for this difference is that the dispersion effect induced by the streaming operator $v\DOT\GRAD$ in the Euclidean space $\RD$ quenches the dissipation effect of 
the Boltzmann collision integral in the large time limit. The present paper uses the Banach fixed point theorem and provides a complete discussion of the large time limit in the case of solutions of the Cauchy problem for the 
Boltzmann equation that are sufficiently close to a global Maxwellian $\MM$, which is in turn assumed ``small'' enough, in the sense that $\nu(\MM)<\tfrac14$. Whether the same asymptotic behavior of the dynamics defined by 
the Boltzmann equation --- especially the fact that the operators $\TT^+$ and $\TT^-$ are onto --- can be established in a more general setting remains an open problem at the time of this writing. Notice however that $\TT^+$ 
and $\TT^-$ are already known to be one-to-one wherever they are defined without smallness assumption on the initial data (Theorem \ref{T-121}), at least in the case of cutoff collision kernels corresponding to soft potentials.


\begin{appendix}

\section{Properties of mild solutions of the Boltzmann equation:\\ Proofs of Theorems B (b) and of H Theorem (c)}


\begin{proof}[Proof of Theorem B (b)]
Let 
$$
\phi(v,x,t):=(a+b\cdot v+c|v|^2+p\cdot(x\wedge v)+q\cdot(x\!-\!tv)+l(x\!-\!tv)\cdot v+m|x\!-\!tv|^2)
$$
where $a,c,l,m\in\R$ and $p,q\in\RD$. Obviously
$$
\phi(v,x+tv,t)=\phi(v,x,0)\quad\hbox{ for each }(v,x,t)\in\RDRDR\,.
$$
Therefore
$$
e^{t_2\ASf}(\phi F)(v,x,t_2)-e^{t_1\ASf}(\phi F)(v,x,t_1)+\int_{t_1}^{t_2}e^{s\ASf}(\phi\BB(F,F))(v,x,s)\,\ds
$$
for a.e. $(v,x)\in\RDRD$ and $t_1,t_2\in I$. Since $0\le F\le\MM$ a.e. on $\RDRD\times I$, one has $\phi F\in L^1(\RDRD\times I)$, and $|\BB(F,F)|\le\AA(\MM)\MM$ a.e. on $\RDRD\times I$ by (\ref{B(F)<A(M)M}). Using 
(\ref{AA(MM)=}), (\ref{aSf=}), (\ref{rho-u=}) and (\ref{theta=}), we conclude that $\phi\BB(F,F)\in L^1(\RDRD\times I)$. Therefore
$$
\begin{aligned}
\iint_\RDRD e^{t_2\ASf}(\phi F)(v,x,t_2)\,\dv\dx=&\iint_\RDRD e^{t_1\ASf}(\phi F)(v,x,t_1)\,\dv\dx\\
&+\int_{t_1}^{t_2}\iint_\RDRD e^{s\ASf}(\phi\BB(F,F))(v,x,s)\,\dv\dx\ds\,,
\end{aligned}
$$
or equivalently
$$
\begin{aligned}
\iint_\RDRD(\phi F)(v,x,t_2)\,\dv\dx=&\iint_\RDRD(\phi F)(v,x,t_1)\,\dv\dx\\
&+\int_{t_1}^{t_2}\iint_\RDRD(\phi\BB(F,F))(v,x,s)\,\dv\dx\ds\,,
\end{aligned}
$$
since 
$$
\iint_\RDRD e^{t\ASf}g(v,x)\,\dv\dx=\iint_\RDRD g(v,x+tv)\,\dv\dx=\iint_\RDRD g(v,y)\,\dv\dee y
$$
for each $g\in L^1(\RDRD)$ and each $t\in\R$. Finally
$$
\int_\RD(\phi\BB(F,F))(v,x,s)\,\dv=0\quad\hbox{ for a.e. }(x,s)\in\RD\times I
$$
because $v\mapsto\phi(v,x,s)$ is a linear combination of $1,v_1,\ldots,v_\D,|v|^2$ (see statement (a) in Theorem B). Hence
$$
\iint_\RDRD(\phi F)(v,x,t_2)\,\dv\dx=\iint_\RDRD(\phi F)(v,x,t_1)\,\dv\dx
$$
for all $t_1,t_2\in\R$.
\end{proof}

\begin{proof}[Proof of H Theorem (c)]
By definition, if $F$ is a mild solution of the Boltzmann equation on $\RDRD\times I$, for a.e. $(v,x)\in\RDRD$, the function $t\mapsto F(v,x,t)$ is absolutely continuous on $I$. Assuming that $\alpha\MM\le F\le\MM$ a.e. 
on $\RDRD\times I$ implies in particular that $F>0$ a.e., so that the chain rule applies (see for instance Corollary VIII.10 in \cite{BrezisAF}) and
$$
\frac{d}{dt}F\ln F(v,x+tv,t)=\BB(F,F)(\ln F+1)(v,x+tv,t)\,,\quad\hbox{ for a.e. }(v,x,t)\in\RDRD\times I\,.
$$
Besides, one has 
$$
\ln\alpha+\ln\MM\le\ln F\le\ln\MM\,,\quad \hbox{ a.e. on }\RDRD\times I\,.
$$
Since $|\BB(F,F)|\le\AA(\MM)\MM$ a.e. on $\RDRD\times I$ by (\ref{B(F)<A(M)M}) and $\ln\MM=O(|x|^2+|v|^2)$ as $|x|+|v|\to\infty$, we conclude that $\BB(F,F)(\ln F+1)\in L^1(\RDRD\times I)$, so that
$$
\begin{aligned}
\frac{d}{dt}H[F](t)&=\frac{d}{dt}\iint_\RDRD F\ln F(v,x+tv,t)\,\dv\dx
\\
&=\iint_\RDRD\BB(F,F)(\ln F+1)(v,x+tv,t)\,\dv\dx
\\
&=\iint_\RDRD\BB(F,F)\ln F(v,y,t)\,\dv\dee y\le 0
\end{aligned}
$$
for a.e. $t\in I$. The last equality follows from the first conservation law in Theorem B (a), and the last inequality from statement (a) in Boltzmann's H Theorem.
\end{proof}

\end{appendix}

\section*{Acknowledgments}
This work has been supported by the NSF under grants  DMS-1109625, and NSF RNMS (KI-Net) grant \#11-07465.   Support from  the Institute of Computational Engineering and Sciences at the University of Texas Austin is gratefully acknowledged.


\end{document}